\documentclass[a4paper, oneside]{amsart}
\usepackage[utf8]{inputenc}

\title{Mapping class groups of $h$-cobordant manifolds}
\author{Samuel Mu\~noz-Ech\'aniz}
\email{sm2600@cam.ac.uk}
\address{Centre for Mathematical Sciences, Wilberforce Road, Cambridge CB3 0WB, United Kingdom}

\usepackage{url}
\usepackage{lmodern,microtype}
\usepackage{amstext,amssymb,mathrsfs,amscd,amsthm,indentfirst}
\usepackage{leftidx}
\usepackage{amsfonts}
\usepackage{enumerate}

\usepackage[dvipsnames,svgnames,x11names,hyperref]{xcolor}
\usepackage[a4paper,width=134mm, top=36mm,bottom=35mm]{geometry}

\usepackage[colorlinks,citecolor=PineGreen,linkcolor=Mahogany,urlcolor=blue,filecolor=Mahogany]{hyperref}

\usepackage[utf8]{inputenc}

\usepackage{subcaption}
\usepackage{thm-restate}
\usepackage{mathtools}
\usepackage{graphicx}
\usepackage{graphics}
\graphicspath{ {images/} }
\usepackage{tikz-cd}
\usepackage{pifont}
\usepackage[inline]{enumitem}
\usepackage{scalerel}

\usepackage[leqno]{amsmath}

\DeclareFontFamily{OT1}{ptm}{}
\DeclareFontShape{OT1}{ptm}{m}{n} { <-> ptmr}{}
\DeclareFontShape{OT1}{ptm}{m}{it}{ <-> ptmri}{}
\DeclareFontShape{OT1}{ptm}{m}{sl}{ <->ptmro}{}
\DeclareFontShape{OT1}{ptm}{m}{sc}{ <-> ptmrc}{}
\DeclareFontShape{OT1}{ptm}{b}{n} { <-> ptmb}{}
\DeclareFontShape{OT1}{ptm}{b}{it}{ <-> ptmbi}{}     
\DeclareFontShape{OT1}{ptm}{bx}{n} {<->ssub * ptm/b/n}{}
\DeclareFontShape{OT1}{ptm}{bx}{it}{<->ssub * ptm/b/it}{}

\DeclareSymbolFont{bold}{OT1}{ptm}{b}{n}
\DeclareMathAlphabet{\mathbf}{OT1}{ptm}{b}{n}  
\DeclareMathAlphabet{\mathrm}{OT1}{ptm}{m}{n}

\DeclareFontFamily{OT1}{psy}{}      
\DeclareFontShape{OT1}{psy}{m}{n}{ <-> s * [0.9] psyr}{}
\DeclareFontFamily{OMS}{ptm}{}     
\DeclareFontShape{OMS}{ptm}{m}{n}{ <8> <9> <10> gen * cmsy }{}
\DeclareFontFamily{OMS}{cmtt}{}     
\DeclareFontShape{OMS}{cmtt}{m}{n}{ <8> <9> <10> gen * cmsy }{}

\SetSymbolFont{operators}{normal}{OT1}{ptm}{m}{n}   
\SetSymbolFont{operators}{bold}{OT1}{ptm}{b}{n}     
\DeclareSymbolFont{emsy}{OT1}{ptm}{m}{it}
\DeclareSymbolFont{emsr}{OT1}{ptm}{m}{n}
\DeclareSymbolFont{emcmr}{OT1}{cmr}{m}{n}   
\DeclareSymbolFont{emsymb}{OT1}{psy}{m}{n}  
\DeclareMathSymbol a{\mathalpha}{emsy}{"61}
\DeclareMathSymbol b{\mathalpha}{emsy}{"62}
\DeclareMathSymbol c{\mathalpha}{emsy}{"63}
\DeclareMathSymbol d{\mathalpha}{emsy}{"64}
\DeclareMathSymbol e{\mathalpha}{emsy}{"65}
\DeclareMathSymbol f{\mathalpha}{emsy}{"66}
\DeclareMathSymbol g{\mathalpha}{emsy}{"67}
\DeclareMathSymbol h{\mathalpha}{emsy}{"68}
\DeclareMathSymbol i{\mathalpha}{emsy}{"69}
\DeclareMathSymbol j{\mathalpha}{emsy}{"6A}
\DeclareMathSymbol k{\mathalpha}{emsy}{"6B}
\DeclareMathSymbol l{\mathalpha}{emsy}{"6C}
\DeclareMathSymbol m{\mathalpha}{emsy}{"6D}
\DeclareMathSymbol n{\mathalpha}{emsy}{"6E}
\DeclareMathSymbol o{\mathalpha}{emsy}{"6F}
\DeclareMathSymbol p{\mathalpha}{emsy}{"70}
\DeclareMathSymbol q{\mathalpha}{emsy}{"71}
\DeclareMathSymbol r{\mathalpha}{emsy}{"72}
\DeclareMathSymbol s{\mathalpha}{emsy}{"73}
\DeclareMathSymbol t{\mathalpha}{emsy}{"74}
\DeclareMathSymbol u{\mathalpha}{emsy}{"75}
\DeclareMathSymbol v{\mathalpha}{emsy}{"76}
\DeclareMathSymbol w{\mathalpha}{emsy}{"77}
\DeclareMathSymbol x{\mathalpha}{emsy}{"78}
\DeclareMathSymbol y{\mathalpha}{emsy}{"79}
\DeclareMathSymbol z{\mathalpha}{emsy}{"7A}
\DeclareMathSymbol A{\mathalpha}{emsy}{"41}
\DeclareMathSymbol B{\mathalpha}{emsy}{"42}
\DeclareMathSymbol C{\mathalpha}{emsy}{"43}
\DeclareMathSymbol D{\mathalpha}{emsy}{"44}
\DeclareMathSymbol E{\mathalpha}{emsy}{"45}
\DeclareMathSymbol F{\mathalpha}{emsy}{"46}
\DeclareMathSymbol G{\mathalpha}{emsy}{"47}
\DeclareMathSymbol H{\mathalpha}{emsy}{"48}
\DeclareMathSymbol I{\mathalpha}{emsy}{"49}
\DeclareMathSymbol J{\mathalpha}{emsy}{"4A}
\DeclareMathSymbol K{\mathalpha}{emsy}{"4B}
\DeclareMathSymbol L{\mathalpha}{emsy}{"4C}
\DeclareMathSymbol M{\mathalpha}{emsy}{"4D}
\DeclareMathSymbol N{\mathalpha}{emsy}{"4E}
\DeclareMathSymbol O{\mathalpha}{emsy}{"4F}
\DeclareMathSymbol P{\mathalpha}{emsy}{"50}
\DeclareMathSymbol Q{\mathalpha}{emsy}{"51}
\DeclareMathSymbol R{\mathalpha}{emsy}{"52}
\DeclareMathSymbol S{\mathalpha}{emsy}{"53}
\DeclareMathSymbol T{\mathalpha}{emsy}{"54}
\DeclareMathSymbol U{\mathalpha}{emsy}{"55}
\DeclareMathSymbol V{\mathalpha}{emsy}{"56}
\DeclareMathSymbol W{\mathalpha}{emsy}{"57}
\DeclareMathSymbol X{\mathalpha}{emsy}{"58}
\DeclareMathSymbol Y{\mathalpha}{emsy}{"59}
\DeclareMathSymbol Z{\mathalpha}{emsy}{"5A}
\DeclareMathSymbol{\bullet}{\mathalpha}{emsymb}{"B7}
\def\Bullet{\leavevmode\unkern{$\m@th\bullet$}\kern.32em\ignorespaces}
\DeclareMathSymbol +{\mathbin}{emcmr}{`+}
\DeclareMathSymbol ={\mathrel}{emcmr}{`=}  
\DeclareMathSymbol{\Gamma}{\mathalpha}{emcmr}{"00}
\DeclareMathSymbol{\Delta}{\mathalpha}{emcmr}{"01}
\DeclareMathSymbol{\Theta}{\mathalpha}{emcmr}{"02}
\DeclareMathSymbol{\Lambda}{\mathalpha}{emcmr}{"03}
\DeclareMathSymbol{\Xi}{\mathalpha}{emcmr}{"04}
\DeclareMathSymbol{\Pi}{\mathalpha}{emcmr}{"05}
\DeclareMathSymbol{\Sigma}{\mathalpha}{emcmr}{"06}
\DeclareMathSymbol{\Upsilon}{\mathalpha}{emcmr}{"07}
\DeclareMathSymbol{\Phi}{\mathalpha}{emcmr}{"08}
\DeclareMathSymbol{\Psi}{\mathalpha}{emcmr}{"09}
\DeclareMathSymbol{\Omega}{\mathalpha}{emcmr}{"0A}
\DeclareMathSizes{7.6}{8}{6}{5}
\DeclareMathAccent{\dot}{\mathalpha}{operators}{"C7} 

\usepackage{booktabs}
\usepackage{verbatim}

\newcommand{\Fun}{\mathrm{Fun}}

\newcommand{\diff}{\mathrm{Diff}}
\newcommand{\bdiff}{\widetilde{\mathrm{Diff}}}

\newcommand{\subcomp}{\mathsf{SubComp}^{\simeq *}}
\newcommand{\SubComp}{\mathsf{SubComp}}
\newcommand{\face}{\mathsf{Face}}
\newcommand{\hcob}{\hspace{2pt}\ensuremath \raisebox{-1pt}{$\overset{h}{\leadsto}$}\hspace{2pt}}
\newcommand{\scob}{\hspace{2pt}\ensuremath \raisebox{-1pt}{$\overset{s}{\leadsto}$}\hspace{2pt}}

\newtheorem{thm}{Theorem}[section]
\newtheorem*{thm*}{Theorem}
\newtheorem*{lem*}{Lemma}

\newtheorem{lem}[thm]{Lemma}
\newtheorem*{cl}{Claim}

\newtheorem*{question*}{\textbf{Question}}
\newtheorem{thmx}{Theorem}
\newtheorem{notn}[thm]{Notation}

\newcommand*\circled[1]{\tikz[baseline=(char.base)]{%
            \node[shape=circle,draw,inner sep=1pt] (char) {#1};}}

\newtheorem{prop}[thm]{Proposition}

\newtheorem{defn}[thm]{Definition}
\theoremstyle{remark}
\newtheorem{exm}[thm]{Example}
\newtheorem{rem}[thm]{Remark}
\newtheorem{digr}[thm]{Digression}
\newtheorem{warn}[thm]{Warning}

\newcommand{\mapsfrom}{\mathrel{\reflectbox{\ensuremath{\longmapsto}}}}

\newcommand{\R}{\mathbb{R}}
\newcommand{\C}{\mathbb{C}}

\newcommand{\Z}{\scaleobj{0.97}{\mathbb{Z}}}
\newcommand{\Q}{\mathbb{Q}}

\newcommand{\im}{\mathrm{Im}\ }
\newcommand{\Wh}{\mathrm{Wh}}
\newcommand{\bdiffh}{\widetilde{\mathrm{Diff}}{\vphantom{\mathrm{Diff}}}^{\hspace{1pt} h}}
\newcommand{\bdiffb}{\widetilde{\mathrm{Diff}}{\vphantom{\mathrm{Diff}}}^{\hspace{1pt}b}}

\newcommand{\Whsp}{\mathbf{Wh}^{\diff}}

\newcommand{\Whsptop}{\mathbf{Wh}^{\mathrm{Top}}}

\newcommand{\Falg}{F^{\mathsf{alg}}}
\newcommand{\Hsp}{\mathbf{H}}

\numberwithin{equation}{section}

\newcommand{\adjunction}{\mathbin{\rotatebox[origin=c]{-90}{$\dashv$}}}

\newsavebox{\pullback}
\sbox\pullback{%
\begin{tikzpicture}%
\draw (0,0) -- (1ex,0ex);%
\draw (1ex,0ex) -- (1ex,1ex);%
\end{tikzpicture}}

\newsavebox{\pushout}
\sbox\pushout{%
\begin{tikzpicture}%
\draw (0,0) -- (0ex,1ex);%
\draw (0ex,1ex) -- (1ex,1ex);%
\end{tikzpicture}}

\begin{document}
\begin{abstract}
We prove that the mapping class group is not an $h$-cobordism invariant of high-dimensional manifolds by exhibiting $h$-cobordant manifolds whose mapping class groups have different cardinalities. In order to do so, we introduce a moduli space of "$h$-block" bundles and understand its difference with the moduli space of ordinary block bundles.
\end{abstract}
\maketitle
\vspace{-25pt}
\section{Introduction}
\subsection{The main result} Automorphism groups of manifolds have been subject to extensive research in algebraic and geometric topology. Inspired by the study of how different $h$-cobordant manifolds can be (see e.g. \cite{JahrenKwasikhcobs, JahrenKwasikInertiaHcobs}), in the present paper we investigate the question of how automorphism groups of manifolds can vary within a fixed $h$-cobordism class. Namely: given an $h$-cobordism $W^{d+1}$ between (closed) smooth\footnote{We will work in the smooth setting for notational preference, but all of the results in this paper are equally valid for the topological and $PL$ categories. See Remarks \ref{smoothcategoryremark} and \ref{diffremarkno2} for modified arguments when $CAT=\mathrm{Top}$ and $PL$.} manifolds $M$ and $M'$ of dimension $d\geq 0$, how different can the homotopy types of the diffeomorphism groups $\diff(M)$ and $\diff(M')$ be?

Certain analogues of this question have led to invariance-type results. Dwyer and Szczarba \cite[Cor. 2]{dwyer} proved that when $d\neq 4$, the rational homotopy type of the identity component $\diff_0(M)\subset \diff(M)$ does not change as $M^d$ varies within a fixed homeomorphism class of smooth manifolds. Krannich \cite[Thm. A]{KrannichExoticSpheres} gave another instance of such a result, showing that when $d=2k\geq 6$ and $M^d$ is closed, oriented and simply connected, the rational homology of $B\diff^+(M)$ in a range is insensitive to replacing $M$ by $M\#\Sigma$, for $\Sigma$ any homotopy $d$-sphere. 

Our main result is, however, that the homotopy types of the diffeomorphism groups of $h$-cobordant manifolds can indeed be different in general. Let $\Gamma(M)$ denote the \textit{mapping class group} of $M$---the group of isotopy classes of diffeomorphisms of $M$, i.e., $\Gamma(M):=\pi_0(\diff(M))$. The \textit{block mapping class group} $\widetilde\Gamma(M)$ is the quotient of $\Gamma(M)$ by the normal subgroup of those classes of diffeomorphisms which are pseudoisotopic to the identity.

\begin{thmx}\label{diffdifference}
In each dimension $d=12k-1\geq 0$, there exist $d$-manifolds $M^d$ (see Theorem \ref{ThmAiequivalent}) $h$-cobordant to the lens space $L=L_7^{12k-1}(r_1:\dots: r_{6k})$, where
$$
r_1=\dots=r_k=1,\qquad r_{k+1}=\dots=r_{2k}=2, \quad \dots \quad r_{5k+1}=\dots=r_{6k}=6\mod 7,
$$
such that
\begin{itemize}
    \item[($i$)] $\widetilde{\Gamma}(L)$ and $\widetilde{\Gamma}(M)$ are finite groups with cardinalities of different $3$-adic valuations,  
    \item[($ii$)] $\Gamma(L)$ and $\Gamma(M)$ are finite groups with cardinalities of different $3$-adic valuations. 
\end{itemize}

\end{thmx}

\begin{rem}
For an oriented connected manifold $M$, there are orientation preserving mapping class groups $\Gamma^+(M)$ and $\widetilde\Gamma^+(M)$, which have index one or two inside the whole mapping class groups $\Gamma(M)$ and $\widetilde{\Gamma}(M)$, respectively. Therefore, the conclusions of Theorem \ref{diffdifference} also hold for $\widetilde{\Gamma}^+(-)$ and $\Gamma^+(-)$.
\end{rem}

\begin{rem}
Theorem \ref{diffdifference}($i$) is the best possible result in the following sense: let $\bdiff(M)$ denote the (geometric realisation of the) semi-simplicial group of \textit{block diffeomorphisms} of $M$ (cf. \cite[p. 20]{BurgLashRoth} or \cite[Defn. 2.1]{RWEbertbdiff}), whose $p$-simplices consist of diffeomorphisms $\phi: M\times \Delta^p\overset{\cong}{\longrightarrow} M\times \Delta^p$ which are face-preserving (i.e. for every face $\sigma\subset \Delta^p$, $\phi$ restricts to a diffeomorphism of $M\times \sigma$). Then we have $\widetilde{\Gamma}(M)=\pi_0(\bdiff((M))$. The restriction map $\rho_M:\bdiff(W)\longrightarrow \bdiff(M)$ is a fibration with fibre $\bdiff_M(W)$, the subspace of block diffeomorphisms of $W$ which fix pointwise a neighbourhood of $M\subset W$. By the $s$-cobordism theorem (see Theorem \ref{scobthm} below), there exists some $h$-cobordism $-W$ from $M'$ to $M$ such that $W\cup_{M'}-W \cong M\times I$ and $-W\cup_{M} W \cong M'\times I$. Then the group homomorphisms
\begin{align*}
\mathrm{Id}_{W}\cup_{M'}-&: \widetilde{C}(M'):=\bdiff_{M'\times\{0\}}(M'\times I)\longrightarrow \bdiff_M(W),\\
\mathrm{Id}_{-W}\cup_{M}-&: \bdiff_M(W)\longrightarrow \bdiff_{M'}(-W\cup_M W)\cong \widetilde{C}(M')
\end{align*}
are easily seen to be homotopy inverse to each other. But the group $\widetilde{C}(M')$ of \textit{block concordances} of $M'$ is contractible (cf. \cite[Lem. 2.1]{BurgLashRoth}), and therefore $\rho_M$ induces an equivalence onto the components that it hits, and similarly for $\rho_{M'}$. In other words, the classifying spaces $B\bdiff(M)$ and $B\bdiff(M')$ share the same universal cover, so
$$
\pi_i(\bdiff(M))\cong \pi_i(\bdiff(M')), \quad i\geq 1.
$$
The upshot is that the homotopy types of $\bdiff(M)$ and $\bdiff(M')$ can at most differ by their sets of path-components, and Theorem \ref{diffdifference}($i$) provides an example showcasing this phenomenon.
\end{rem}

\subsection{\texorpdfstring{Moduli spaces of $h$- and $s$-block bundles}{Moduli spaces of h- and s-block bundles}} Recall that for $d\geq 5$, the \textit{Whitehead group} $\Wh(M)$ of a compact $d$-manifold $M$ (see Section \ref{WhTsection}) classifies isomorphism classes of $h$-cobordisms starting at $M$. This group has an involution denoted $\tau\mapsto \overline \tau$ which, roughly speaking and up to a factor of $(-1)^d$, corresponds to reversing the direction of an $h$-cobordism (see (\ref{hcobdualformula})).

In Section \ref{section3} we will introduce the $h$- and $s$-\textit{block moduli spaces}, $\widetilde{\mathcal M}^h$ and $\widetilde{\mathcal M}^s$ respectively, whose vertices (as semi-simplicial sets) are the smooth closed $d$-manifolds, for some fixed integer $d\geq 0$. A path in the former (resp. latter) space between $d$-manifolds $M$ and $M'$ is exactly an $h$-cobordism $W: M\hcob M'$ (resp. an $s$-cobordism $W: M\scob M'$, i.e., an $h$-cobordism with vanishing Whitehead torsion (see Section \ref{scobsection})). The $s$-block moduli space $\widetilde{\mathcal M}^s$ is, somewhat in disguise, a well-known object; in Proposition \ref{connectedcompsofMs} we identify the path-component of $M^d$ in $\widetilde{\mathcal M}^s$ with $B\bdiff(M)$, the classifying space for the group of block diffeomorphisms of $M$.

The second main result we state arises as part of the proof of Theorem \ref{diffdifference}, but may be of independent interest: there is a natural inclusion $\widetilde{\mathcal M}^s\xhookrightarrow{}\widetilde{\mathcal M}^h$ which forgets the simpleness condition. We identify the homotopy fibre of this inclusion (i.e. the homotopical difference between the $h$- and $s$-block moduli spaces) as a certain infinite loop space.

\begin{thmx}\label{ThmB}
Let $M$ be a closed $d$-dimensional manifold, and let $C_2:=\{e,t\}$ act on the Whitehead group $\Wh(M)$ by $t\cdot\tau:=(-1)^{d-1}\overline\tau$. Write $H\Wh(M)$ for the \textit{Eilenberg--MacLane} spectrum associated to $\Wh(M)$, and let $H\Wh(M)_{hC_2}:=H\Wh(M)\wedge_{C_2}(EC_2)_+$ stand for the homotopy $C_2$-orbits of $H\Wh(M)$. For $d\geq 5$, there is a homotopy cartesian square
$$
\begin{tikzcd}
\Omega^\infty(H\Wh(M)_{hC_2})\rar\dar\arrow[dr, phantom, "\usebox\pullback" , very near start, color=black]  & {\widetilde{\mathcal M}^s}\dar[hook]\\
\{M^d\}\rar & {\widetilde{\mathcal M}^h},
\end{tikzcd}
$$
where the lower horizontal map is the inclusion of $M$ as a point in $\widetilde{\mathcal M}^h$.
\end{thmx}
As we will explain in Section \ref{Rothsection}, this result is intimately tied to the \textit{Rothenberg exact sequence} \cite[Prop. 1.10.1]{Ranickiexactseq}.

\subsection*{Structure of the paper} 
Section \ref{section2} serves as a reminder to the reader of the $s$-\textit{cobordism theorem} and some of the properties of \textit{Whitehead torsion}. 

In Section \ref{section3} we prove Theorem \ref{ThmB}. The proof boils down to arguing that certain simplicial abelian group $\Falg_\bullet(A)$ corresponds to the spectrum $HA_{hC_2}$ under the Dold--Kan correspondence (see Theorem \ref{ActualThmC}).

Section \ref{section4} deals with part $(i)$ of Theorem \ref{diffdifference}, which is proved in Theorem \ref{ThmAiequivalent}. We analyse the lower degree part of the homotopy long exact sequence associated to the homotopy pullback square of Theorem \ref{ThmB}. The proof of Theorem \ref{diffdifference}$(ii)$ builds on part $(i)$ and pseudoisotopy theory, and comprises Section \ref{section5}.

Appendix \ref{AppendixA} is an algebraic $K$-theory computation required for Sections \ref{section4} and \ref{section5}. Appendix \ref{appendixB} explores the connection between Theorem \ref{ThmB} and the theory of Weiss--Williams \cite{WWI}.

\subsection*{Acknowledgements}
The author is immensely grateful to his Ph.D. supervisor, Oscar Randal-Williams, for suggesting this problem and for many illuminating discussions that have greatly benefited this work. He would also like to thank Connor Malin for originally raising the question in \cite{MathOverflowHCob}, the anonymous referee for valuable feedback and corrections, John Nicholson for making him aware of \cite{Zhang2019}, and Sander Kupers, Robin Stoll and Henry Wilton for further helpful comments. The author was supported by an EPSRC Ph.D. Studentship (grant no. 2597647).

\section{Notation and recollections}\label{section2}
All manifolds will be assumed to be compact and smooth (possibly with corners).
\subsection{Whitehead Torsion}\label{WhTsection} The \textit{Whitehead group} of $(\pi,w)$ \cite[$\S$6]{MilnorWhiteheadTorsion}, where $\pi$ is a group and $w: \pi\to C_2=\{\pm1\}$ is a homomorphism, is the abelian group
$$
\Wh(\pi,w):=GL(\Z\pi)^{\mathrm{ab}}/(\pm\pi)
$$
equipped with the following involution: the anti-involution on the group ring $\Z\pi$ given by
\begin{equation}\label{Zpi1antiinvolution}
    a=\sum_{g\in \pi}a_g\cdot g\longmapsto \overline a:=\sum_{g\in \pi}w(g)\hspace{1pt} a_g\cdot g^{-1}, \quad a_g\in \Z,
\end{equation}
induces an involution on $\Wh(\pi,w)$ by sending an element represented by a matrix $\tau=(\tau_{ij})$ to its conjugate transpose $\overline\tau:=(\overline{\tau}_{ji})$. We will refer to this involution as the \textit{algebraic involution} on $\Wh(\pi,w)$. We will write $\Wh(\pi)$ for $\Wh(\pi,w)$ if $w$ is the trivial homomorphism, or if we are simply disregarding this involution. If $X$ is a finite CW-complex with a choice of basepoint in each of its connected components, the \textit{Whitehead group} of $X$ is
$$
\Wh(X):=\bigoplus_{X_j\in \pi_0(X)}\Wh(\pi_1(X_j)).
$$
If $X=M$ is moreover a manifold, the algebraic involution on $\Wh(M)$ is that induced by $w=w_1(M)\in H^1(M;\Z/2)$, the first Stiefel--Whitney class of $M$. Since every inner automorphism of a group $\pi$ induces the identity on $\Wh(\pi)$ \cite[Lem. 6.1]{MilnorWhiteheadTorsion}, the Whitehead group $\Wh(X)$ does not depend (up to canonical isomorphism) on the choice of basepoint in each path component of $X$; for this reason, we shall ignore basepoints from now on. 

Given a homotopy equivalence between finite pointed CW-complexes $f: X\overset{\simeq}\longrightarrow Y$, we will denote by $\tau(\hspace{2pt}f)\in \Wh(X)$ its (\textit{Whitehead}) \textit{torsion} \cite[$\S$7]{MilnorWhiteheadTorsion}. It only depends on $f$ up to homotopy \cite[Lem. 7.7]{MilnorWhiteheadTorsion}. Let us collect a few properties of the Whitehead torsion $\tau(-)$ that we will use throughout the paper:

\begin{itemize}[leftmargin=9pt, itemsep=4pt]
    \item \textit{Composition rule}: $\tau(-)$ is a crossed homomorphism in the sense that if $f:X\overset\simeq\longrightarrow Y$ and $g: Y\overset\simeq\longrightarrow Z$ are homotopy equivalences, then \cite[Lem. 7.8]{MilnorWhiteheadTorsion}
\begin{equation}\label{whiteheadtorsioncomposition}
\tau(g\circ f)=\tau(\hspace{2pt}f)+f_*^{-1}\tau(g),
\end{equation}
where $f_*: \Wh(X)\overset{\cong}{\longrightarrow} \Wh(Y)$ is the natural isomorphism induced by $\pi_1(\hspace{2pt}f): \pi_1(X)\overset{\cong}{\longrightarrow} \pi_1(Y)$. 

\item \textit{Inclusion-exclusion principle}: if $X=X_0\cup X_1$ and $Y=Y_0\cup Y_1$, where $X_0$, $X_1$, $Y_0$, $Y_1$, $X_{01}:=X_0\cap X_1$ and $Y_{01}:=Y_{0}\cap Y_1$ are all finite CW-complexes, and 
$$
f_0: X_0\overset\simeq\longrightarrow Y_0, \quad f_1: X_1\overset\simeq\longrightarrow Y_1,\quad f_{01}=f_0\cap f_1: X_{01}\overset\simeq\longrightarrow Y_{01},
$$
are homotopy equivalences, then the torsion of the homotopy equivalence $f=f_0\cup f_1: X\overset{\simeq}\longrightarrow Y$ is \cite[Thm. 23.1]{Cohen1973-zw}
\begin{equation}\label{whiteheadtorsincexc} 
\tau(\hspace{2pt} f)=(i_0)_*\tau(\hspace{2pt} f_0)+(i_1)_*\tau(\hspace{2pt} f_1)-(i_{01})_*\tau(\hspace{2pt} f_{01})\in\Wh(X),
\end{equation}
where $i_0: X_0\xhookrightarrow{}X$, $i_1: X_1\xhookrightarrow{}X$ and $i_{01}: X_{01}\xhookrightarrow{}X$ are the inclusions.

\item \textit{Product rule}: $\tau(-)$ is multiplicative with respect to the Euler characteristic in the sense that for any homotopy equivalence $f: X\overset\simeq \longrightarrow Y$ and any finite connected CW-complex $K$ with basepoint $*\in K$ \cite[Thm. 23.2]{Cohen1973-zw},
\begin{equation}\label{eulerchartors}
\tau(\hspace{2pt} f\times \mathrm{id}_K)=\chi(K)\cdot i_*\tau(\hspace{2pt} f)\in \Wh(X\times K), 
\end{equation}
where $i: X\cong X\times \{*\}\xhookrightarrow{}X\times K$ is the inclusion.
\end{itemize}

A homotopy equivalence $f$ as above is said to be \textit{simple}, and denoted $f: X\overset{\simeq_s}\longrightarrow Y$, if $\tau(\hspace{2pt}f)=0$. We will write $s\mathrm{Aut}(X)\subset h\mathrm{Aut}(X)$ for the topological submonoid (see (\ref{whiteheadtorsioncomposition})) of simple homotopy automorphisms of $X$.

\subsection{\texorpdfstring{The $s$-cobordism theorem}{The s-cobordism theorem}}\label{scobsection} Let $M^d$ be a smooth compact manifold of dimension $d$. A \textit{cobordism from} $M$ \textit{rel} $\partial M$ is a triple $(W; M, M')$, also written as $W: M \leadsto M'$, consisting of a $(d+1)$-manifold $W^{d+1}$  with boundary
$$
\partial W\cong M\cup M'\cup (\partial M\times [0,1])
$$
so that $M\cap (\partial M\times [0,1])=\partial M\times\{0\}$ and $M'\cap (\partial M\times [0,1])=\partial M\times \{1\}$ (in particular $\partial M'=\partial M$). Cobordisms are often accompanied with an additional data of \textit{collars}, i.e., open neighbourhoods of $M$ and $M'$ in $W$ diffeomorphic to $M\times [0,\epsilon)$ and $M'\times (1-\epsilon,1]$ (rel $\partial M\times I$) for some small $\epsilon>0$, but the choice of such is contractible. If $\partial M=\emptyset$, this coincides with the usual notion of a cobordism between closed manifolds. Such a cobordism is called an \textit{h-cobordism} if the inclusions $i_M: (M,\partial M)\to (W, \partial M\times I)$ and $i_{M'}: (M',\partial M')\to (W, \partial M\times I)$ are homotopy equivalences of pairs. In such case we will write $W: M\hcob M'$ to emphasise that $W$ is an $h$-cobordism from $M$ to $M'$. The \textit{torsion} of $W$ with respect to $M$ is
$$
\tau(W,M):=\tau(i_M)\in \Wh(M).
$$
If $\tau(W,M)=0$, such an $h$-cobordism $W:M\hcob M'$ is said to be \textit{simple} (or an \textit{$s$-cobordism}), and denoted $W: M\scob M'$. This definition does not depend on the direction of $W$ since the torsion of an $h$-cobordism satisfies the \textit{duality formula} \cite[$\S$10]{MilnorWhiteheadTorsion}
\begin{equation}\label{hcobdualformula}
    \tau(W,M')=(-1)^{d}(h^W)_*\overline{\tau(W,M)}.
\end{equation}
Here $h^W: M\simeq M'$ is the \textit{natural homotopy equivalence}
\begin{equation}\label{natheqW}
\begin{tikzcd}
h^W: M\rar[hook, "i_M", "\simeq"'] &W\rar[two heads,"r_{M'}", "\simeq"'] &M',
\end{tikzcd}
\end{equation}
where $r_{M'}$ is some homotopy inverse to $i_{M'}$ (so $h^W$ is only well-defined up to homotopy).

Due to the composition rule (\ref{whiteheadtorsioncomposition}), the torsion of an $h$-cobordism is nearly additive with respect to composition: namely if $W:M\hcob M'$ and $W': M'\hcob M''$ are $h$-cobordisms, we write $W'\circ W: M\hcob M''$ for the $h$-cobordism $W\cup_{M'}W'$, which can be made smooth by pasting along collars. Then
\begin{equation}\label{torsofhcobcomposition}
    \tau(W'\circ W,M)=\tau(W,M)+(h^W)_*^{-1}\tau(W',M').
\end{equation}

Let $h\mathrm{Cob}_\partial(M)$ denote the set of $h$-cobordisms rel boundary starting at $M$, up to diffeomorphism rel $M$. We will use the following a great deal \cite{Mazur1963, Bardenscob}.

\begin{thm}[$s$-Cobordism Theorem rel boundary]\label{scobthm}
If $d=\dim M\geq 5$, then there is a bijection
$$
h\mathrm{Cob}_\partial(M)\longleftrightarrow \mathrm{Wh}(M), \qquad 
(W: M\hcob M')\longmapsto \tau(W,M).
$$

\end{thm}

\section{The block moduli spaces of manifolds}\label{section3}
As explained in the introduction, we now present the $h$- and $s$-block moduli spaces of manifolds, in which a path, i.e., a $1$-simplex, is an $h$- or $s$-cobordism, respectively. To describe what higher-dimensional simplices should be we give the next definition, which is inspired by \cite[$\S$2]{RWLuckCollar}.

\begin{defn}\label{stratifieddefn}
Fix once and for all some small $\epsilon>0$. A compact smooth manifold with corners $W^{d+p}\subset \R^\infty\times \Delta^p$ is said to be \textbf{stratified over} $\Delta^p$ if:
\begin{itemize}
\item[($i$)] $W$ is a closed manifold if $p=0$,

\item[($ii$)] $W$ is transverse to $\R^{\infty}\times \sigma$ for every proper face $\sigma\subset \Delta^p$ and $W_\sigma:=W\cap (\R^\infty\times \sigma)$ is a $(d+\dim\sigma)$-dimensional manifold stratified over $\Delta^{\dim\sigma}\cong \sigma$,

\item[($iii$)] $W$ satisfies the $\epsilon$-collaring conditions of \cite[Defn. 2.3.1(ii)]{RWLuckCollar}.
\end{itemize}
We will write $W^{d+p}\Rightarrow \Delta^p$ for such a manifold. A map $f: W\to V$ between manifolds stratified over $\Delta^p$ is said to be \textbf{face-preserving}, and denoted $f: W\to_\Delta V$, if for every face $\sigma\subset \Delta^p$ we have $f(W_\sigma)\subset V_\sigma$ and $f$ satisfies the collaring conditions of \cite[Defn. 2.3.1(iii)]{RWLuckCollar} (roughly, $f$ must be the product $f_\sigma\times \mathrm{Id}$ in the $\epsilon$-neighbourhood of the strata $W_\sigma$, where $f_\sigma:=f\mid_{W_\sigma}$). If moreover $f_\sigma$ is a homotopy equivalence, simple homotopy equivalence or diffeomorphism for all $\sigma\subset \Delta^p$, we will write $f: W\overset{\spadesuit}\longrightarrow_\Delta V$ for $\spadesuit=\ \simeq_h,\ \simeq_s$ or $\cong$, respectively.
\end{defn}

\begin{notn}\label{simplicialnotation} Let $\Lambda_i^p\subset \Delta^p$ denote the $i$-th horn of $\Delta^p$ ($i=0,\dots,p$).
\begin{itemize}
    \item If $0\leq i_0<\dots<i_r\leq p$, we write $\langle i_0,\dots,i_r\rangle\subset \Delta^p$ for the face spanned by the vertices $i_0,\dots, i_r\in \Delta^p$. 
    
    \item If $W$ is stratified over $\Delta^p$, we will often write $\partial_i W$ for $W_{\langle0,\dots,\widehat{i},\dots,p\rangle}$ and $W_i$ for $W_{\langle i\rangle}$. For instance, $\langle0\dots,\widehat{i\hspace{1pt}},\dots,p\rangle\equiv\partial_i\Delta^p\subset \Delta^p$. 
    
    \item If $K\subset \Delta^p$ is a simplicial sub-complex, we will write $W_K$ for $W\cap (\R^\infty\times K)$. In the particular case that $K=\Lambda_i^p$, we set $
    \Lambda_i(W):=W_{\Lambda_i^p}
    $. For instance, if $\sigma\subset \Delta^p$ is some face, $\Lambda_i(\sigma)$ denotes the $i$-th horn of $\sigma$ ($i=0,\dots,\dim\sigma$).
    
    \item  If $f: W\longrightarrow_\Delta V$ is face-preserving, we will write $\partial_i\hspace{2pt} f$ for $f_{\partial_i\Delta^{p}}=f\mid_{\partial_iW}$.
\end{itemize}

\end{notn}

\begin{exm}
A cobordism $W^{d+1}: M\leadsto M'$ between closed manifolds $M$ and $M'$ is always diffeomorphic to a manifold $W'\subset \R^\infty\times \Delta^1$ stratified over $\Delta^1$ with $W'_0\cong M$ and $W'_1\cong M'$.
\end{exm}

\begin{defn}\label{modulispacesdefn}
Fix some integer $d\geq 0$. The \textbf{$h$-block moduli space of $d$-manifolds} is the semi-simplicial set $\widetilde{\mathcal M}^h_\bullet$ with $p$-simplices
\begin{equation}\label{hmodspacepsimplex}
\widetilde{\mathcal M}^h_p:=\left\{\begin{array}{cc}
   W^{d+p}\\
   \Downarrow\\
   \Delta^p
\end{array}
:\exists\ f: W\overset{\simeq_h}\longrightarrow_\Delta W_0\times \Delta^p\right\},
\end{equation}
and with face maps given by restriction to face-strata
$$
\partial_i: \widetilde{\mathcal M}^h_p\longrightarrow \widetilde{\mathcal M}^h_{p-1}, \quad \begin{array}{cc}
   W^{d+p}\\
   \Downarrow\\
   \Delta^p
\end{array}\longmapsto \begin{array}{cc}
   \partial_iW^{d+p}\\
   \Downarrow\\
   \partial_i\Delta^p\cong \Delta^{p-1},
\end{array}\quad i=0,\dots,p.
$$

The \textbf{$s$-block moduli space of $d$-manifolds} $\widetilde{\mathcal M}^s_\bullet$ is its \textit{simple} analogue, where $\simeq_h$ in (\ref{hmodspacepsimplex}) is replaced by $\simeq_s$, and has a natural inclusion $\widetilde{\mathcal M}^s_\bullet\xhookrightarrow{}\widetilde{\mathcal M}^h_\bullet$. We will let $\widetilde{\mathcal M}^{h}$ and $\widetilde{\mathcal M}^s$ denote the geometric realisations $|\widetilde{\mathcal M}^{h}_\bullet|$ and $|\widetilde{\mathcal M}^{s}_\bullet|$, respectively.
\end{defn}

\begin{rem}
Let $M$ be a closed $d$-manifold and $\widetilde{\mathcal M}^s_\bullet(M)$ denote the path-component of $M$ in $\widetilde{\mathcal M}^s_\bullet$. Our definition of $\widetilde{\mathcal M}^s_\bullet(M)$ differs from that of $\mathcal{M}(M)_\bullet$ in \cite{RWLuckCollar} in that our condition (ii) in Definition \ref{stratifieddefn} is stronger than condition (i) of Definition 2.3.1 loc. cit.; there, it is only required that $W$ be transverse to $\R^\infty\times \sigma$ for faces of the form $\sigma=\partial_i \Delta^p$. As noted right after \cite[Defn. 2.3.1]{RWLuckCollar}, their defined $\mathcal{M}(M)_\bullet$ is Kan, and our stronger requirement does not affect this Kan condition as any proper subface of $\Delta^p$ that is not of the form $\partial_i\Delta^p$ is already a subface of any horn $\Lambda_j^p$. Thus, $\widetilde{\mathcal M}^s_\bullet$ (and similarly $\widetilde{\mathcal M}^h_\bullet$) is Kan.

\end{rem}

The next two subsections are devoted to prove Theorem \ref{ThmB}. But first, we study the $s$-block moduli space $\widetilde{\mathcal M}^s$ more closely. We recall that the classifying space $B\bdiff(M)$ for the simplicial group of block diffeomorphisms has a semi-simplicial model (see e.g. \cite{RWEbertbdiff}) in which the $p$-simplices are
$$
B\bdiff(M)_p=\left\{\begin{array}{cc}
   W^{d+p}\\
   \Downarrow\\
   \Delta^p
\end{array}
:\exists\ \phi: W\overset{\cong}\longrightarrow_\Delta M\times \Delta^p\right\},
$$
and therefore there is a forgetful inclusion $B\bdiff(M)\xhookrightarrow{} \widetilde{\mathcal M}^s$.
\begin{prop}\label{connectedcompsofMs}
For $d\geq 5$, there is a decomposition of $\widetilde{\mathcal M}^s$ into connected components
\begin{equation}\label{blockmodspacebdiff}
\widetilde{\mathcal M}^s=\bigsqcup_{\substack{[M^d] \ \text{up}\\ \text{to $s$-cob.}}} B\bdiff(M)=\bigsqcup_{\substack{[M^d] \ \text{up}\\ \text{to diffeo.}}} B\bdiff(M). 
\end{equation}
\end{prop}

In order to prove this proposition, we will need the following simple observation.

\begin{lem}\label{lemmasurgeryscob}
Let $W^{d+p}\Rightarrow \Delta^p$ represent a $p$-simplex in $\widetilde{\mathcal M}^s_\bullet$. For every face inclusion $\xi\subset\sigma\subset \Delta^p$, the map $W_\xi\xhookrightarrow{}W_\sigma$ is a simple homotopy equivalence. In particular if $p=1$, $W$ is an $s$-cobordism from $W_{0}$ to $W_{1}$.
\end{lem}
\begin{proof}
 Let $f:W\overset{\simeq_s}\longrightarrow_\Delta W_0\times \Delta^p$ be some face-preserving simple homotopy equivalence. The inclusion $W_\xi\xhookrightarrow{}W_\sigma$ is homotopic to a composition of simple maps 
 $$
\begin{tikzcd}
W_\xi\rar["f_\xi"', "\simeq_s"] &W_0\times \xi\ar[r,hook, "\simeq_s", "\text{by (\ref{eulerchartors})}"'] &W_0\times \sigma\rar["f_\sigma^{-1}"', "\simeq_s"]& W_\sigma,
\end{tikzcd}
$$
where $f_\sigma^{-1}$ is any homotopy inverse to $f_\sigma$. Therefore it is also simple.
\end{proof}

\begin{proof}[Proof of Proposition \ref{connectedcompsofMs}]
For a closed manifold $M^d$, let $\widetilde{\mathcal M}^s_\bullet(M)$ denote the path-component of $M$ in $\widetilde{\mathcal M}^s_\bullet$. We only have to argue that $\widetilde{\mathcal M}^s(M)\subset B\bdiff(M)$, which is the following claim when $r=-1$.
\begin{cl}
Let $W\in \widetilde{\mathcal M}^s_p(M)$ and suppose that for some $-1\leq r\leq p-1$, there exist face-preserving diffeomorphisms
$$
\phi_i: \partial_iW\overset{\cong}\longrightarrow_\Delta M\times\Delta^{p-1}, \quad 0\leq i\leq r,
$$
such that $\partial_i\phi_j=\partial_{j-1}\phi_{i}$ for $0\leq i<j\leq r$. Then there exists some face-preserving diffeomorphism $\phi: W\overset{\cong}\longrightarrow_\Delta M\times \Delta^p$ such that $\partial_i\phi=\phi_i$ for $0\leq i\leq r$. In particular $W\in B\bdiff(M)_p$.
\end{cl}
Proceed by induction on $p\geq 0$. The statement is vacuous when $p=0$, and it holds by Lemma \ref{lemmasurgeryscob} and the $s$-cobordism theorem when $p=1$. Suppose that the claim is true for $p-1\geq 0$. By the induction hypothesis, we can find diffeomorphisms $\phi_i: \partial_iW\overset{\cong}\longrightarrow_\Delta M\times \Delta^{p-1}$ for $0\leq i\leq p-1$ such that $\partial_i\phi_j=\partial_{j-1}\phi_{i}$ for $0\leq i<j\leq p-1$. By pasting these together, we obtain a (face-preserving) diffeomorphism $\Lambda_p(\phi): \Lambda_p(W)\overset{\cong}\longrightarrow_\Delta M\times \Lambda_p^{p}$. Now by Lemma \ref{lemmasurgeryscob} and the inclusion-exclusion principle (\ref{whiteheadtorsincexc}), the inclusion $\Lambda_p(W)\xhookrightarrow{} W$ is a simple homotopy equivalence. Unbending the corners of $\Lambda_p(W)$, the $s$-cobordism theorem for manifolds with boundary (Theorem \ref{scobthm}) provides a face-preserving diffeomorphism $\phi: W\overset{\cong}\longrightarrow_\Delta M\times\Delta^p$ extending $\Lambda_p(\phi)$, as required. 
\end{proof}

By analogy to (\ref{blockmodspacebdiff}), we define $B\bdiffh(M)$ to be the path-component of $M$ in $\widetilde{\mathcal M}^h$, and so obtain a decomposition 
\begin{equation}\label{connectedcompsofMh}
\widetilde{\mathcal M}^h=\bigsqcup_{\substack{[M^d] \ \text{up}\\ \text{to $h$-cob.}}} B\bdiffh(M).
\end{equation}
\begin{rem}
The semi-simplicial sets $B\bdiff(M)_\bullet$ and $B\bdiffh(M)_\bullet$ have $M$ as their natural basepoint. If $M$ and $M'$ are $s$-cobordant, i.e. diffeomorphic (resp. $h$-cobordant), then $B\bdiff(M)$ and $B\bdiff(M')$ (resp. $B\bdiffh(M)$ and $B\bdiffh(M')$) are the same semi-simplicial set but equipped with different basepoints. 
\end{rem}

\begin{digr}\label{bdiffbounded}
Let $G: \mathsf{sSet}_*\to \mathsf{sGrp}$ denote the \textit{Kan simplicial loop space functor} \cite[$\S$10 and 11]{KanLoopGroup}. As we will see in Remark \ref{degeneraciesBdiffh}, the semi-simplicial set $B\bdiffh(M)_\bullet$ can be upgraded to a simplicial set. The simplicial group $\bdiffh(M):=GB\bdiffh(M)$ has been studied in previous literature under different names \cite[Appendix 5]{WWI}. More precisely, if $d\geq 5$ then $\bdiffh(M^d)$ is weakly equivalent to $\bdiffb(M\times \R)$, the space of block diffeomorphisms of $M\times \R$ \textit{bounded} in the $\R$-direction. This will be proved in Proposition \ref{appendixBprop} of Appendix \ref{appendixB}. With this in mind, the computation of the homotopy groups of $\bdiffb(M\times\R)/\bdiff(M)$ in \cite[Cor. 5.5]{WWI} agrees with Theorem \ref{ThmB}.
\end{digr}

\subsection{\texorpdfstring{A simplicial model for $H(-)_{hC_2}$}{A simplicial model for H(-)hC2}}
Let $A$ be a $\Z[C_2]$-module, i.e., an abelian group equipped with a $\Z$-linear involution $a\mapsto a^*:=t\cdot a$, where $t\in C_2$ denotes the generator. Write $HA$ for the \textit{Eilenberg--MacLane spectrum} associated to $A$, and let $HA_{hC_2}:=HA\wedge_{C_2}(EC_2)_+$ denote the homotopy $C_2$-orbits of $HA$. In this section we present a simplicial model
$$
\Falg_\bullet(-): \mathsf{Mod}_{\Z[C_2]}\longrightarrow \mathsf{sAb}
$$
for the functor $H(-)_{hC_2}: \mathsf{Mod}_{\Z[C_2]}\longrightarrow H\Z\text{-$\mathsf{Mod}$}$ in the following sense. 
\begin{thm}\label{ActualThmC}
Let $\Omega^\infty\text{-$\mathsf{Top}$}$ denote the category of infinite loop spaces. There is a natural equivalence
$$
|\Falg_\bullet(-)|\simeq \Omega^\infty(H(-)_{hC_2}):\mathsf{Mod}_{\Z[C_2]}\longrightarrow \Omega^\infty\text{-$\mathsf{Top}$},
$$
i.e., there is a zig-zag of natural weak equivalences connecting the left and the right hand functors.
\end{thm}

\subsubsection{The simplicial abelian group $\Falg_\bullet(A)$} We now define $\Falg_\bullet(A)$ as an \textit{algebraic analogue} of the semi-simplicial set $F_\bullet(M)$ (see (\ref{geomFddefn}) and Proposition \ref{Fdgeomvsalg}) when $A=\Wh(M)$ with the $C_2$-action $t\cdot \tau:=(-1)^{d-1}\overline \tau$ of Theorem \ref{ThmB}. We will need some preliminaries first.

A \textit{simplicial sub-complex} of $\Delta^p$ is a collection\footnote{We allow the empty collection $\emptyset=\{\hspace{3pt}\}$.} $K$ of non-empty subsets $\sigma$ of $[\hspace{1pt}p]=\{0,\dots,p\}$ such that if $\xi\subset \sigma$ and $\sigma\in K$, then $\xi\in K$ too. The \textit{realisation} of a subset $\sigma\subset [\hspace{1pt}p]$ is the subspace 
$$
|\sigma|:=\{(t_0,\dots,t_p)\in \Delta^p: \text{$t_i=0$ if $i\notin \sigma$} \}\subset \Delta^p.
$$
Then, the \textit{realisation} of a simplicial sub-complex $K$ of $\Delta^p$ is the subspace
$$
|K|:=\bigcup_{\sigma\in K}|\sigma|\subset \Delta^p.
$$
Since $|K|=|K'|$ if and only if $K=K'$, we will often identify a simplicial sub-complex $K$ of $\Delta^p$ with its realisation $|K|$.

Let $\SubComp_p$ denote the poset of simplicial complexes of $\Delta^p$, ordered by inclusion of collections. The assignment $[\hspace{1pt}p]\mapsto \SubComp_p$ assembles into a cosimplicial poset $\SubComp_\bullet$ in the obvious way: given an order-preserving arrow $a:[\hspace{1pt}p]\to [q]$, we set
$$
a: \SubComp_p\longrightarrow \SubComp_q, \quad K\mapsto a(K):=\{a(\sigma)\subset [q]: \sigma\in K\}.
$$
On realisations, we have $|a(K)|=a(|K|)\subset \Delta^q$, where the second $a$ now denotes the map $\Delta^p\to \Delta^q$ coming from the cosimplicial space $\Delta^\bullet$. So, for instance, the $i$-th coface map $\partial^{i}: \SubComp_{p-1}\to\SubComp_{p}$ identifies $\Delta^{p-1}$ with the $i$-th face of $\Delta^p$, i.e., $\partial^{i}\Delta^{p-1}:=\partial_i\Delta^p$.

More important to us will be the subposet $\subcomp_p\subset \SubComp_p$ consisting of those sub-complexes whose realisation is contractible (in particular non-empty). The assignment $[\hspace{1pt}p]\mapsto \subcomp_p$ now assembles only into a \emph{semi}-cosimplicial poset, for the coface maps $\partial^{i}: \SubComp_{p-1}\to\SubComp_{p}$ send $\subcomp_{p-1}$ to $\subcomp_p$, but the codegeneracies fail to do so: for instance, the codegeneracy $s^0: \SubComp_{3}\to\SubComp_{2}$ sends the contractible sub-complex of $\Delta^3$ cosisting of the edges $\{0,2\}$, $\{2,3\}$ and $\{1,3\}$ (and all its subsets) to $\partial\Delta^2$. 

Seen as a category, $\subcomp_p$ admits all pushouts, for the pushout of contractible simplicial complexes is also contractible. For an (abelian) group $A$ (seen as a one-object groupoid), write $\mathrm{Fun}(\subcomp_p, A)$ for the set of functors $\tau:\subcomp_p\to A$ and $\mathrm{Fun}^\square(\subcomp_p, {A})\subset \mathrm{Fun}(\subcomp_p,A)$ for the subset consisting of those functors $\tau$ such that for any pushout square
\begin{equation}\label{firstpushoutsquare}
\begin{tikzcd}
K_{01}\rar\dar\arrow[dr, phantom, "\usebox\pushout" , very near end, color=black]  &K_1\dar\\
K_0\rar &K
\end{tikzcd}
\end{equation}
in $\subcomp_p$:
\begin{equation}\label{funsquarecondition}
\tau(K_{01}\to K_1)=\tau(K_0\to K)\in A.
\end{equation}
Equivalently, by taking the transpose of (\ref{firstpushoutsquare}),
$$
\tau(K_{01}\to K_0)=\tau(K_1\to K).
$$

\begin{notn}
If $K,L\in \subcomp_p$ with $K\leq L$, there is a unique arrow $K\to L$. We will write $\tau(L,K)\in A$ for $\tau(K\to L)$ resembling the notation for $h$-cobordisms. Note that $\tau(K,K)=0$ for every $K$. We will sometimes refer to $\tau(L,K)$ as a \textbf{torsion element}.
\end{notn}

\begin{lem}\label{incexclemma}
Let $\tau:\subcomp_p\to A$ be a functor. Then $\tau\in \mathrm{Fun}^\square(\subcomp_p,A)$ if and only if for every diagram in $\subcomp_p$ of the form
\begin{equation}\label{pushoutdiagramsubcomp}
    \begin{tikzcd}
K_{01}\rar\dar\arrow[dr, phantom, "\usebox\pushout" , very near end, color=black]  &K_1\dar \arrow[ddr, bend left]&\\
K_0\arrow[drr, bend right=21pt]\rar &K\arrow[dr] &\\
&&L,
\end{tikzcd}
\end{equation}
the functor $\tau$ satisfies the inclusion-exclusion principle (compare to (\ref{whiteheadtorsincexc})):
\begin{equation}\label{incexcprinciple}
\tau(L,K)=\tau(L,K_0)+\tau(L,K_1)-\tau(L,K_{01}).
\end{equation}
\end{lem}
\begin{proof}
If $\tau\in \mathrm{Fun}^\square(\subcomp_p,A)$, then
\begin{align*}
    \tau(L,K_0)+\tau(L,K_1)-\tau(L,K_{01})&=\tau(L,K)+\tau(K,K_0)+\tau(L,K_1)-\tau(L,K_1)-\tau(K_1,K_{01})\\
    &=\tau(L,K),
\end{align*}
where the second line follows from (\ref{funsquarecondition}). Conversely (\ref{funsquarecondition}) follows from the inclusion-exclusion principle (\ref{incexcprinciple}) applied to the diagram (\ref{pushoutdiagramsubcomp}) with $L=K$, noting that $\tau(K,K)=0$.
\end{proof}

Observe that both $\mathrm{Fun}(\subcomp_p,A)$ and $\mathrm{Fun}^\square(\subcomp_p,A)$ are abelian groups under morphism-wise addition. Therefore, $\mathrm{Fun}(\subcomp_\bullet,A)$ defines a semi-simplicial abelian group whose $i$-th face map is $\partial^{\mathrm{Fun}}_i:=\mathrm{Fun}(\partial^{i}, A)$. These face maps clearly descend to $\mathrm{Fun}^\square(\subcomp_\bullet,A)$, and we will write $\partial_i^\square\equiv\partial_i^{\mathrm{Fun}}\mid_{\mathrm{Fun}^\square}$ for their restriction. We will now construct a system of degeneracies $s_i^\square$ for $\mathrm{Fun}^\square(\subcomp_\bullet,A)$ compatible with the $\partial_i^\square$'s which makes it into a simplicial abelian group---this will be handy when invoking the Dold--Kan correspondence later in Section \ref{DoldKanSection} (see also Remark \ref{technicalrem}).

\begin{rem}
    One might consider the abelian groups $\mathrm{Fun}(\SubComp_p,A)$ and $\mathrm{Fun}^\square(\SubComp_p,A)$, defined analogously. Then, $\mathrm{Fun}(\SubComp_\bullet,A)$ defines an actual simplicial abelian group whose $i$-th face and degeneracy maps are $\partial_i^{\Fun}$ and $s_i^\Fun=\Fun(s^i,A)$, respectively. However, $\mathrm{Fun}^\square(\SubComp_\bullet,A)$ is only semi-simplicial: for any non-zero $\tau\in \mathrm{Fun}^\square(\SubComp_1,A)$, $s_1^{\mathrm{Fun}}\tau$ does not satisfy (\ref{funsquarecondition}) for $K=\Lambda_2^2$, $K_0=\partial_0\Delta^2$, $K_1=\partial_1\Delta^2$ and $K_{01}=\langle2\rangle$.
\end{rem}

Let $\face_\bullet$ denote the sub-\emph{cosimplicial} poset of $\SubComp_\bullet$ consisting of those sub-complexes of $\Delta^\bullet$ which are \textit{faces}---those collections $K$ for which there exists some $\sigma\subset[\bullet]$ such that $\xi\in K$ for every $\xi\subset \sigma$; thus, the empty sub-complex $\emptyset\subset \Delta^\bullet$ is \emph{not} a face. Note that $\face_\bullet\subset\subcomp_\bullet$ and write $\iota$ for the inclusion. The following result says that a functor in $\mathrm{Fun}^\square(\subcomp_\bullet,A)$ is completely determined by the torsion elements corresponding to face inclusions.

\begin{prop}\label{SubFacelem}
There is a natural isomorphism of semi-simplicial abelian groups
$$
\iota^*: \mathrm{Fun}^\square(\subcomp_\bullet, A)\cong \mathrm{Fun}(\face_\bullet, A): \iota_!.
$$
Therefore, $\mathrm{Fun}^\square(\subcomp_\bullet, A)$ upgrades to a simplicial abelian group with degeneracy maps
$$
s_i^\square:=\iota_!\circ\mathrm{Fun}(s^{i},A)\circ\iota^*: \mathrm{Fun}^\square(\subcomp_\bullet,A)\to \mathrm{Fun}^\square(\subcomp_{\bullet+1},A).
$$
\end{prop}
\begin{proof}
    We begin with the definition of the map $\iota_!$. For $\tau\in \mathrm{Fun}(\face_p,A)$, let us first define $\iota_!\tau(L,K)$ for any inclusion of sub-complexes $K\subset L\subset \Delta^p$ in $\subcomp_p$. Setting
\begin{equation}\label{functorialityofiota!}
\iota_!\tau(L,K):=\iota_!\tau(\Delta^p,K)-\iota_!\tau(\Delta^p,L),
\end{equation}
it suffices to specify $\iota_!\tau(\Delta^p,K)$ for any $K\in \subcomp_p$. Note that, as defined in \eqref{functorialityofiota!}, $\iota_!\tau$ is immediately a functor $\subcomp_p\to A$ (even for arbitrary values of $\iota_!\tau(\Delta^p,K)$ for $K\subsetneq \Delta^p$), since if $K\subset L\subset M$ are sub-complexes of $\Delta^p$, then
\begin{align*}
\iota_!\tau(M,K)&=\iota_!\tau(\Delta^p,K)-\iota_!\tau(\Delta^p,M)\\
&=\iota_!\tau(\Delta^p,K)-\iota_!\tau(\Delta^p,L)+\iota_!\tau(\Delta^p,L)-\iota_!\tau(\Delta^p,M)\\
&=\iota_!\tau(L,K)+\iota_!\tau(M,L).
\end{align*}
We now define $\iota_!\tau(\Delta^p,K)$ for \emph{any\footnote{Not only contractible ones!}} sub-complex $K$ of $\Delta^p$ as 
\begin{equation}\label{IotaTauDefn}
\iota_!\tau(\Delta^p, K):=\sum_{\sigma\in K}T(\Delta^p,\sigma), \quad \text{where}\quad T(\Delta^p,\sigma):=\sum_{\emptyset\neq\xi\subseteq \sigma}(-1)^{\dim \sigma+\dim \xi} \tau(\Delta^p,\xi).    
\end{equation}
Here $\dim\sigma$ denotes the dimension of the face $\sigma$ (one less than the cardinality of $\sigma$ viewed as a subset of $[\hspace{1pt}p]$). We now establish the following properties of the functor $\iota_!\tau$:

\begin{enumerate}[itemsep=4pt]
    \item[(i)]  $\iota_!\tau$ satisfies the inclusion-exclusion principle \eqref{incexcprinciple} and hence, by Lemma \ref{incexclemma}, $\iota_!\tau$ restricted to $\subcomp_p$ is an element of $\Fun^\square(\subcomp_p,A)$.

    \item[(ii)] $\iota^*\iota_!\tau\equiv\tau$ as functors $\face_p\to A$.

    \item[(iii)] For every $0\leq i\leq p$, $\partial_i(\iota_!\tau)\equiv \iota_!(\partial_i\tau)$ as functors $\subcomp_{p-1}\to A$.
\end{enumerate}

To establish property (i), note that, in light of \eqref{functorialityofiota!}, it suffices to show that $\iota_!\tau$ satisfies the inclusion-exclusion for diagrams (\ref{pushoutdiagramsubcomp}) where $L=\Delta^p$. This follows by a simple check:
\begin{align*}
    \iota_!\tau(\Delta^p,K)&= \sum_{\sigma\in K}T(\Delta^p,\sigma)\\
    &= \sum_{\sigma\in K_0}T(\Delta^p,\sigma)+\sum_{\sigma\in K\setminus K_0}T(\Delta^p,\sigma)\\
    &=\sum_{\sigma\in K_0}T(\Delta^p,\sigma)+\sum_{\sigma\in K_1\setminus K_{01}}T(\Delta^p,\sigma)\\
    &=\sum_{\sigma\in K_0}T(\Delta^p,\sigma)+\sum_{\sigma\in K_1}T(\Delta^p,\sigma)-\sum_{\sigma\in K_{01}}T(\Delta^p,\sigma)\\
    &=\iota_!\tau(\Delta^p,K_0)+\iota_!\tau(\Delta^p,K_1)-\iota_!\tau(\Delta^p,K_{01}).
\end{align*}

To prove (ii), it suffices yet again to show that $\iota_!\tau(\Delta^p,\sigma)=\tau(\Delta^p,\sigma)$ for every face $\sigma\in \face_p$. Given a linear combination $X$ of $\tau(\Delta^p,\xi)$'s for $\xi\in \face_p$, let $X_{(\eta)}$ denote the coefficient of $\tau(\Delta^p,\eta)$ in $X$. We now count such coefficient $\iota_!\tau(\Delta^p,\sigma)_{(\eta)}$ for any face $\eta\in \face_p$. Clearly $\iota_!\tau(\Delta^p,\sigma)_{(\eta)}=0$ if $\eta$ is not contained in $\sigma$, so assume that $\eta\subset \sigma$. Then,
\begin{align*}
    \iota_!\tau(\Delta^p,\sigma)_{(\eta)}&=\sum_{\xi\in \sigma}T(\Delta^p,\xi)_{(\eta)}=\sum_{\eta\subset \xi\in \sigma}T(\Delta^p,\xi)_{(\eta)}=\sum_{\eta\subset \xi\in \sigma}(-1)^{\dim \xi+\dim\eta}\\
    &=\sum_{i=0}^{\dim\sigma-\dim\eta}(-1)^i\binom{\dim\sigma-\dim\eta}{i}=\left\{
\begin{array}{cl}
    1, & \dim\eta=\dim \sigma, \\
     0,& \text{otherwise.} 
\end{array}
    \right.
\end{align*}
Since $\dim\eta=\dim\sigma$ if and only if $\eta=\sigma$, it follows that $\iota_!\tau(\Delta^p,\sigma)=\tau(\Delta^p,\sigma)$, as claimed.

For (iii), again we only need to show that both functors $\partial_i(\iota_!\tau)$ and $\iota_!(\partial_i\tau)$ agree on arrows in $\subcomp_{p-1}$ of the form $K\to \Delta^{p-1}$. Then, 
\begin{align*}
    \iota_!(\partial_i\tau)(\Delta^{p-1},K)&=\sum_{\sigma\in K}\sum_{\emptyset\neq\xi\subset \sigma}(-1)^{\dim\sigma+\dim\xi}\partial_i\tau(\Delta^{p-1},\xi)\\
    &=\sum_{\sigma\in K}\sum_{\emptyset\neq\xi\subset \sigma}(-1)^{\dim\sigma+\dim\xi}\tau(\partial^i\Delta^{p-1},\partial^i\xi)\\
    &=\sum_{\sigma\in K}\sum_{\emptyset\neq\xi\subset \sigma}(-1)^{\dim\sigma+\dim\xi}\left(\tau(\Delta^{p},\partial^i\xi)-\tau(\Delta^p,\partial^i\Delta^{p-1})\right)\\
    &=\iota_!\tau(\Delta^p,\partial^i K)-\left(\sum_{\sigma\in K}\sum_{\emptyset\neq\xi\subset\sigma}(-1)^{\dim\sigma+\dim\xi}\right)\cdot\tau(\Delta^p,\partial^i\Delta^{p-1})
\end{align*}
Now by a similar consideration as in (ii), the sum $\sum_{\emptyset\neq\xi\subset\sigma}(-1)^{\dim\sigma+\dim\xi}$ is $(-1)^{\dim\sigma}$ (as we are not summing over $\xi=\emptyset$). Thus, using this and property (ii), we see that
\begin{equation}\label{EulerCharEq}
\iota_!(\partial_i\tau)(\Delta^{p-1},K)=\iota_!\tau(\Delta^p,\partial^i K)-\chi_K\cdot \iota_!\tau(\Delta^p,\partial^i\Delta^{p-1}),
\end{equation}
where $\chi_K$ stands for the Euler characteristic of $K$. As \emph{$K$ is contractible} (so $\chi_K=1$), we obtain
$$
\iota_!(\partial_i\tau)(\Delta^{p-1},K)=\iota_!\tau(\partial^i\Delta^{p-1},\partial^i K)=\partial_i(\iota_!\tau)(\Delta^{p-1},K),
$$
as desired.

By definition, $\iota_!: \Fun(\face_p,A)\to \Fun^\square(\subcomp_p, A)$ is clearly a group homomorphism, so, by (iii), we have successfully constructed a morphism of semi-simplicial abelian groups
$$
\iota_!: \mathrm{Fun}(\face_\bullet, A)\longrightarrow \mathrm{Fun}^\square(\subcomp_\bullet, A),
$$
and by (ii), we also have that $\iota_!$ is a right inverse of $\iota^*$. To finish the proof of the proposition, it only remains to show that $\iota^*$ is injective.

To this end, suppose that $\tau\in \Fun^\square(\subcomp_p,A)$ is a functor such that $\iota^*\tau\equiv 0$. We show that $\tau(\Delta^p,K)=0$ for every $K\in \subcomp_p$ (and hence $\tau\equiv 0$), by induction on the \textit{dimension} of $K$, that is, the maximal dimension of a face of $\Delta^p$ contained in $K$. This is clear if $\dim K=0$, so assume the claim holds for every $K'\in \subcomp_p$ of dimension $\leq j-1$, and let $K\in \subcomp_p$ be of dimension $j\geq 1$. Suppose that there exists a $j$-dimensional face $\sigma\subset K$ and some $0\leq i\leq j$ such that $\partial_i\sigma$ is not contained in any other $j$-dimensional face $\sigma'$ of $K$. Then, consider the sub-complex $K':=K\setminus(\operatorname{Int}\sigma\cup \operatorname{Int}\partial_i\sigma)$ and the diagram
$$
\begin{tikzcd}
\Lambda_i(\sigma)\rar\dar\arrow[dr, phantom, "\usebox\pushout" , very near end, color=black]  &K'\dar \arrow[ddr, bend left]&\\
\sigma\arrow[drr, bend right=21pt]\rar &K\arrow[dr] &\\
&&\Delta^p.
\end{tikzcd}
$$
Since $\Lambda_i(\sigma)\to\sigma$ is an equivalence and the square is a pushout, it follows that $K'\in \subcomp_p$. Then, by induction on the number of $j$-dimensional faces of $K$, we may assume that $\tau(\Delta^p,K')=0$, and hence, by the inclusion-exclusion principle (\ref{incexcprinciple}) for the above diagram,
$$
\tau(\Delta^p,K)=\tau(\Delta^p,\sigma)+\tau(\Delta^p,K')-\tau(\Delta^p,\Lambda_i(\sigma))=0+0-0=0.
$$
The penultimate equality follows from our assumption $\iota^*\tau\equiv 0$ and from our induction hypothesis as $\Lambda_i(\sigma)$ is $(j-1)$-dimensional. But such $\sigma$ and $0\leq i\leq j$ must always exist. For if it did not, then 
$$
M:=\bigcup_{\sigma\in K:\ \dim\sigma=j}\sigma
$$
would be a non-empty, closed PL-manifold of dimension $j$, and hence $H_j(M;\Z/2)\neq 0$. But since $K$ is obtained from $M$ by attaching faces of dimension $\leq j-1$, then $H_j(M;\Z/2)\xhookrightarrow{}H_j(K;\Z/2)$, leading to a contradiction as $K$ is contractible. This concludes the proof of the proposition.
\end{proof}

\begin{rem}
    As part of the proof, we see that, for each $p\geq 0$, the group homomorphism $\iota_!:\Fun(\face_p,A)\to \Fun^\square(\subcomp_p,A)$ actually factors through $\Fun^\square(\SubComp_p,A)$. However, 
    $$
\iota_!: \Fun(\face_\bullet,A)\longrightarrow \Fun^\square(\SubComp_\bullet,A)
    $$
    does \textit{not} assemble to a morphism of semi-simplicial abelian groups by (\ref{EulerCharEq}); surprisingly, though, if $\SubComp_\bullet^{\chi = 1}$ denotes the sub-poset of those sub-complexes with Euler characteristic one, then
    $$
\iota_!: \Fun(\face_\bullet,A)\longrightarrow \Fun^\square(\SubComp_\bullet^{\chi = 1},A)
    $$
    is indeed a semi-simplicial isomorphism that factors the one of Proposition \ref{SubFacelem}.
\end{rem}

A functor $\tau\in\mathrm{Fun}(\subcomp_p,A)$ will be said to satisfy \textit{face-horn duality} for $\sigma$ if
\begin{equation}\label{facehornduality}
    \tau(\sigma, \partial_i\sigma)=(-1)^{\dim\sigma}\tau^*(\sigma,\Lambda_i(\sigma)), \quad i=0,\dots,\dim\sigma.
\end{equation}
In the above notation, $\tau^*(L,K)$ stands for $(\tau(L,K))^*$. Write $D_p(A)\subset \mathrm{Fun}(\subcomp_p,A)$ for the subgroup of functors that satisfy face-horn duality for every face $\sigma\subset \Delta^p$, and let $D_\bullet(A)\subset \mathrm{Fun}(\subcomp_\bullet,A)$ denote the corresponding sub-semi-simplicial abelian group.

\begin{rem}
If $\tau\subset D_p(A)\cap\mathrm{Fun}^\square(\subcomp_p,A)$, then $\tau$ satisfies more general sorts of dualities. For instance, if $\sigma\subset\Delta^p$ is a face and $0\leq i<j\leq \dim\sigma$, then
$$
\tau(\sigma, \partial_i\sigma\cup\partial_j\sigma)=(-1)^{\dim\sigma}\tau^*(\sigma,\partial\sigma\setminus\operatorname{Int}(\partial_i\sigma\cup\partial_j\sigma)).
$$
This follows from the inclusion-exclusion principle (\ref{incexcprinciple}) of Lemma \ref{incexclemma} applied to $K=\partial_i\sigma\cup\partial_j\sigma$ and $\Lambda_i(\sigma)=(\partial\sigma\setminus \operatorname{Int}(\partial_i\sigma\cup\partial_j\sigma))\cup\partial_j\sigma$, and $L=\sigma$. By induction, one can generalise this duality to any proper collection of faces $\partial_I\sigma:=\bigcup_{i\in I}\partial_i\sigma$, $I\subsetneq\{0,\dots,\dim\sigma\}$:
\begin{equation}\label{generalisedfacehornduality}
\tau(\sigma,\partial_I\sigma)=(-1)^{\dim\sigma}\tau^*(\sigma,\partial_J\sigma), \quad J:=\{0,\dots,\dim\sigma\}\setminus I.
\end{equation}
Even more generally, if $K\in\subcomp_p$ is a union of $k$-dimensional faces and $Q\subset \partial K$ is a contractible sub-complex which is a union of $(k-1)$-dimensional faces, then
\begin{equation}\label{generalisedfacehornduality2}
\tau(K,Q)=(-1)^{k}\tau^*(K,\overline{\partial K\setminus Q}).
\end{equation}
\end{rem}
We now give a simple inductive criterion to check if a functor satisfies all face-horn dualities.
\begin{lem}\label{dualitycriterionlem}
Let $\tau\in \mathrm{Fun}^{\square}(\subcomp_p,A)$ satisfy face-horn duality for all $\sigma\subsetneq \Delta^p$ and for the $0$-th face-horn of $\Delta^p$:
$$
\tau(\Delta^p,\partial_0\Delta^p)=(-1)^{p}\tau^*(\Delta^{p},\Lambda^p_0).
$$
Then $\tau$ satisfies face-horn duality for $\Delta^p$ too, i.e., $\tau\in D_p(A)$.
\end{lem}
\begin{proof}
For $i\in \{1,\dots,p\}$ denote $\Lambda^p_{0i}$ for $\partial_{\{1,\dots,\widehat{i},\dots,p\}}\Delta^p$, and consider the two pushout diagrams in $\subcomp_p$
$$
\begin{array}{cc}
\begin{tikzcd}
\Lambda_{0}(\partial_i\Delta^p)\rar\dar\arrow[dr, phantom, "\usebox\pushout" , very near end, color=black]  &\Lambda_{0i}^p\dar \arrow[ddr, bend left]&\\
\partial_i\Delta^p\arrow[drr, bend right = 18pt]\rar &\Lambda_0^p\arrow[dr] &\\
&&\Delta^p,
\end{tikzcd}
   &  
\begin{tikzcd}
\Lambda_{i-1}(\partial_0\Delta^p)\rar\dar\arrow[dr, phantom, "\usebox\pushout" , very near end, color=black]  &\partial_0\Delta^p\dar \arrow[ddr, bend left]&\\
\Lambda_{0i}^p\arrow[drr, bend right = 17pt]\rar &\Lambda_i^p\arrow[dr] &\\
&&\Delta^p.
\end{tikzcd} 
\end{array}
$$
Note that $\Lambda_{i-1}(\partial_0\Delta^p)$ is the union of the codimension one sub-faces of $\partial_0\Delta^p=\langle1,\dots,p\rangle\subset \Delta^p$ that contain the $i$-th vertex of $\Delta^p$. We check directly that $\tau$ satisfies duality for the $i$-th face-horn using the inclusion-exclusion principle (\ref{incexcprinciple}).
\begin{align*}
\tau(\Delta^p,\partial_i\Delta^p)&=\tau(\Delta^p,\Lambda_0^p)+\tau(\Lambda_0^p, \partial_i\Delta^p)\\
&=(-1)^p\tau^*(\Delta^p,\partial_0\Delta^p)+\tau(\Lambda_{0i}^p,\Lambda_0(\partial_i\Delta^p))\\
&=(-1)^p\tau^*(\Delta^p,\Lambda_i^p)+(-1)^p\tau^*(\Lambda_i^p,\partial_0\Delta^p)+(-1)^{p-1}\tau^*(\Lambda_{0i}^p,\Lambda_{i-1}(\partial_0\Delta^p))\\
&=(-1)^p\tau^*(\Delta^p, \Lambda_i^p).
\end{align*}
In the third line we have used (\ref{generalisedfacehornduality2}) for $K=\Lambda_{0i}^p$ and $Q=\Lambda_0(\partial_i\Delta^p)$. 
\end{proof}

Finally, we write $Z_p(A)\subset \mathrm{Fun}(\subcomp_{p+1},A)$ for the subgroup of functors $\tau$ such that
$$
\tau(L,K)=0, \quad \forall\ K\subset L\subset \partial_0\Delta^{p+1}=\langle1,\dots,p+1\rangle.
$$
The assignment $[\hspace{1pt} p]\mapsto Z_p(A)$ defines a semi-simplicial abelian group $Z_\bullet(A)$ whose $i$-th face map is the restriction of $\partial_{i+1}^{\mathrm{Fun}}$ to $Z_p(A)$.

\begin{defn}
The simplicial abelian group $\Falg_\bullet(A)\subset \mathrm{Fun}(\subcomp_{\bullet+1}, A)$ has as $p$-simplices
$$
\Falg_p(A):=Z_p(A)\cap D_{p+1}(A)\cap\mathrm{Fun}^\square(\subcomp_{p+1},A),
$$
as face maps $\delta_i: \Falg_p(A)\to \Falg_{p-1}(A)$ the restriction to $\Falg_p(A)$ of $\partial_{i+1}^{\mathrm{Fun}}$, and as degeneracy maps $s_i: \Falg_p(A)\to \Falg_{p+1}(A)$ the restriction to $\Falg_p(A)$ of the map $s_{i+1}^{\square}$ from Proposition \ref{SubFacelem}.
\end{defn}

\begin{lem}\label{FalgIsSimplicial}
$\Falg_\bullet(A)$ as defined above is a simplicial abelian group.
\end{lem}
\begin{proof}
The only non-trivial thing to check is that $s_i$ sends $D_{p}(A)$ into $D_{p+1}(A)$, so let $\tau\in D_p(A)\cap \mathrm{Fun}^\square(\subcomp_p,A)$. Without loss of generality assume $i=0$, and by the induction hypothesis and the simplicial identities, it suffices to check that $s_0\tau$ satisfies face-horn duality for the top face $\Delta^{p+1}$. By Lemma \ref{dualitycriterionlem}, just checking this for the $0$-th face-horn of $\Delta^{p+1}$ will suffice. Since $s_0\tau$ satisfies the inclusion-exclusion principle,
\begin{alignat*}{3}
    s_0\tau(\Delta^{p+1},\Lambda_0^{p+1})&=\sum_{k=1}^{p+1}(-1)^{k-1}&&\sum_{0<j_1<\dots<j_k\leq p+1}s_0\tau\left(\Delta^{p+1}, \bigcap_{r=1}^k\partial_{j_r}\Delta^{p+1}\right)\\
    &=\sum_{k=1}^{p+1}(-1)^{k-1}\Bigg\{&&\sum_{1<j_1<\dots<j_k}\tau(\Delta^p, \langle 0,\dots,\widehat{j_1-1},\dots, \widehat{j_k-1},\dots, p+1\rangle)\\
     & &&+\sum_{1=j_1<j_2<\dots<j_k}\hspace{-4pt}\tau(\Delta^p, \langle 0,\dots,\widehat{j_2-1},\dots, \widehat{j_k-1},\dots, p+1\rangle)\Bigg\}\\
     &=(-1)^p&&\hspace{-40pt}\sum_{\underbrace{\scriptstyle 1<j_1<\dots<j_{p+1}\leq p+1}_{=\emptyset}}\tau(\Delta^p, \langle 0,\dots,\widehat{j_1-1},\dots, \widehat{j_k-1},\dots, p+1\rangle)=0.
\end{alignat*}
In the other hand, $s_0\tau(\Delta^{p+1},\partial_0\Delta^{p+1})=\tau(\Delta^p,\Delta^p)=0=(-1)^{p+1}(s_0\tau(\Delta^{p+1},\Lambda_0^{p+1}))^*$, as required.
\end{proof}

\subsubsection{Proof of Theorem \ref{ActualThmC}}\label{DoldKanSection}
Recall that the \textit{Dold--Kan correspondence} \cite[$\S$III.2, Cor. 2.3]{GoerssJardine} establishes an equivalence of categories
\begin{equation}\label{DoldKan}
\begin{tikzcd}
N: \mathsf{sAb}\rar[yshift={1ex}, "\adjunction"'] &\lar[yshift={-1ex}]\mathsf{Ch}_{\geq 0}(\Z):\Gamma,
\end{tikzcd}
\end{equation}
where $N$ is the \textit{normalised Moore complex} functor, given for a simplicial group $G=(G_\bullet,\delta_\bullet)$ by
$$
(NG)_n:=\bigcap_{i=1}^n\ker\left(\delta_i: G_n\to G_{n-1}\right), \quad d_n=\delta_0\mid_{(NG)_n}: (NG)_n\longrightarrow (NG)_{n-1}.
$$
Under (\ref{DoldKan}), we will identify $\Falg_\bullet(A)$ with the (connective) chain complex $\mathbb{A}_{hC_2}$ given by
$$
\begin{tikzcd}
\dots\rar["1-t"] &A\rar["1+t"] &A  \rar["1-t"] &A\rar&0=(\mathbb{A}_{hC_2})_{-1}.
\end{tikzcd}
$$

\begin{prop}\label{quasiisoprop}
The map $\psi_\bullet^{A}: (N\Falg(A)_\bullet, d_\bullet)\longrightarrow \mathbb{A}_{hC_2}$ given by
\begin{equation}\label{psiAdefn}
\psi^{A}_n: N\Falg(A)_n\longrightarrow (\mathbb{A}_{hC_2})_n=A,\quad \tau\longmapsto \tau(\Delta^{n+1},\langle0\rangle),
\end{equation}
is a quasi-isomorphism of chain complexes. In particular,
$$
\pi_n(\Falg_\bullet(A))\cong H_n(N\Falg(A))\cong H_n(C_2;A)=\left\{
\begin{array}{cc}
    \frac{A}{\left\{b- b^*\ \mid \ b\in A\right\}}, & n=0, \\[10pt]
    \frac{\left\{a\in A\ \mid \ a=(-1)^{n+1} a^*\right\}}{\left\{b+(-1)^{n+1} b^*\ \mid\ b\in A\right\}}, & n\geq 1. 
    \end{array}
\right. 
$$
\end{prop}

\begin{proof}
First we verify that $\psi_\bullet=\psi_\bullet^{A}$ is a chain map. Let $\tau\in N\Falg(A)_n$ and write 
$$
a:=\psi_n(\tau)=\tau(\Delta^{n+1},\langle0\rangle),\quad b:=\psi_{n-1}(d_n\tau)=\tau(\partial_1\Delta^{n+1},\langle0\rangle), \quad c:=\tau(\Delta^{n+1},\partial_1\Delta^{n+1}).
$$
Noting that $\tau(\Lambda_1^{n+1},\langle0\rangle)=0$ by inclusion-exclusion, applying $\tau$ to the diagram in $\subcomp_{n+1}$
$$
\begin{tikzcd}
\langle0\rangle\rar\dar&\Lambda_1^{n+1}\dar\\
\partial_1\Delta^{n+1}\rar &\Delta^{n+1}
\end{tikzcd}
$$
and duality of $\tau$, we obtain
$$
(-1)^{n+1} c^*=a=b+c\implies a+(-1)^{n} a^*=b+c+(-1)^{n}\left((-1)^{n+1} c^*\right)^*=b,
$$
i.e., $d_n(\psi_n(\tau))=\psi_{n-1}(d_n\tau)$.

We now have to show that the map
$$
\psi_*: H_n(N\Falg(A))\longrightarrow H_n(C_2;A)
$$
is an isomorphism for $n\geq 0$. 

\begin{cl}
For $n>0$, there is a bijection 
\begin{align*}
\tau_{(-)}:\{a\in A\ \mid \ a=(-1)^{n+1} a^*\}&\longleftrightarrow \bigcap_{i=0}^n\ker(\delta_{i}: \Falg_n(A)\longrightarrow \Falg_{n-1}(A)):\psi_n, \label{Fdcyclebijection}\\
a&\longmapsto \tau_{a}\\
\tau(\Delta^{n+1},\langle0\rangle)&\mapsfrom \tau
\end{align*}
where $\tau_{a}\in \Falg_n(A)$ is the functor given by
$$
\tau_{a}(L,K)=\left\{
\begin{array}{cl}
    a, & \text{if}\ K\subsetneq L=\Delta^{n+1},\\
    0, & \text{otherwise.}
\end{array}
\right.
$$
\end{cl}
\begin{proof}[Proof of Claim]
Note that the condition $a=(-1)^{n+1}a^*$ is exactly the face-horn duality for $\Delta^{n+1}$, so $\tau_{a}$ is indeed an element of $\Falg_n(A)$. Also observe that $\psi_n(\tau_a)=a$, so we only need to show that $\tau_{(-)}$ is surjective. Let $\tau$ be a cycle in $\Falg_n(A)$, and set $a:=\tau(\Delta^{n+1},\langle0\rangle)$; we check that $\tau=\tau_{a}$. By the functoriality relation $\tau(L,K)=\tau(\Delta^{n+1},K)-\tau(\Delta^{n+1},L)$, we may assume that $L=\Delta^{n+1}$, and by Proposition \ref{SubFacelem} that $K=\sigma\in \face_{n+1}$. Let $i$ be such that $\sigma\subset\partial_i\Delta^{n+1}$. As $\partial_i\tau=0$,
$$
\tau(\Delta^{n+1},\sigma)=\tau(\Delta^{n+1},\partial_i\Delta^{n+1})+\tau(\partial_i\Delta^{n+1},\sigma)=\tau(\Delta^{n+1},\partial_i\Delta^{n+1}),
$$
so it is enough to show that $\tau(\Delta^{n+1},\partial_i\Delta^{n+1})=a$ for $i=0,\dots,n+1$. If $i\neq 0$, this follows from the definition of $a:=\tau(\Delta^{n+1},\langle0\rangle)$. Applying $\tau$ to the diagram in $\subcomp_{n+1}$
$$
\begin{tikzcd}
\langle2,\dots,n+1\rangle\rar\dar &\partial_1\Delta^{n+1}\dar\\
\partial_0\Delta^{n+1}\rar &\Delta^{n+1}
\end{tikzcd}
$$
and noting that $\partial_0\tau=0$, we obtain $\tau(\Delta^{n+1},\partial_0\Delta^{n+1})=a$, as required. Also by definition $\tau_a$ is a cycle in $\Falg_\bullet(A)$ and $\tau_{a}(\Delta^{n+1},\langle0\rangle)=a$, so the claim follows. 
\end{proof}

The previous claim shows that $\psi_*$ is surjective when $n>0$. But this is also the case when $n=0$, as
$$
\psi_0: N\Falg_0(A)=\Falg_0(A)\longrightarrow A, 
$$
is an isomorphism: namely, for $a\in A$, the functor $\tau: \subcomp_1\to A$ given by $\tau(\Delta^1,\langle0\rangle):=a$ and $\tau(\Delta^1,\langle 1\rangle):= -a^*$ (and zero otherwise) is clearly in $\Falg_0(A)$ and sent to $a$ under $\psi_0$. Conversely, if $\tau\in \Falg(A)_0$ and $\tau(\Delta^1,\langle0\rangle)=0$, duality then forces $\tau$ to be zero itself.

For injectivity of $\psi_*$, let $\tau\in N\Falg(A)_n$ be a cycle such that $\psi_*[\tau]=0$, i.e., $\tau(\Delta^{n+1},\langle0\rangle)=b+(-1)^{n+1} b^*$ for some $b\in A$. It is not difficult to see that there exists a functor $T\in \mathrm{Fun}^\square(\subcomp_{n+2},A)$ with
$$
\partial_i T=\left\{
\begin{array}{cc}
    \tau, & i=1,  \\
    0, & i\neq 1,
\end{array}
\right.\qquad
T(\Delta^{n+2},\partial_i\Delta^{n+2}):=\left\{
\begin{array}{cc}
    -b, & i=1, \\
    (-1)^{n+1} b^*, & i\neq 1,
\end{array}
\right.\quad (0\leq i\leq n+2).
$$
By construction, $T$ satisfies face-horn duality for any face $\sigma\neq \Delta^{n+1}$ and for the first face-horn $(\partial_1\Delta^{n+2},\Lambda_1^{n+2})$. Therefore by Lemma \ref{dualitycriterionlem} it satisfies all face-horn dualities. Then $T$ is clearly an element of $N\Falg(A)_{n+1}$ bounding $\tau$, so $[\tau]=0$ in $H_n(N\Falg(A)_\bullet)$. This finishes the proof.
\end{proof}

\begin{rem}
The map $\psi_\bullet^{A}: N\Falg(A)\overset{\cong}\longrightarrow \mathbb{A}_{hC_2}$ is in fact an isomorphism of chain complexes. An element $\tau\in N\Falg(A)_n$ is completely determined by $\delta_0\tau(=\partial_1\tau)$ and $b:=\tau(\Delta^{n+1},\langle0\rangle)\in A$ using functoriality and duality. By the claim in the proof of Proposition \ref{quasiisoprop}, $\delta_0\tau=\tau_{a}$ for some $a\in A$ with $a=(-1)^{n} a^*$. Face-horn duality for $(\partial_1\Delta^{n+1},\Lambda_{1}^{n+1})$ yields
$$
a+(-1)^{n+1} b^*=b\implies a=b+(-1)^{n} b^*,
$$
so $\delta_0\tau=\tau_{a}$ is completely determined by $b=\tau(\Delta^{n+1},\langle0\rangle)$. As we will not need this fact, we leave it to the reader to check that $\psi_\bullet^{A}$ is indeed surjective.
\end{rem}

Before moving on to the proof of Theorem \ref{ActualThmC}, we still need some categorical background. For the rest of the section, we shall adopt the conventions of \cite{Schwede}---for instance, our category $\Omega^\infty\text{-$\mathsf{Top}$}$ of infinite loop spaces (a.k.a. connective spectra) is modelled by the model category of $\Gamma$-spaces thereof. For $G=\{e\}$ or $C_2$ and $\mathsf{C}$ a category, the category of $G$-objects in $\mathsf{C}$ is $\mathsf{C}^G:=\mathrm{Fun}(G,\mathsf{C})$. Observe that there are natural isomorphisms of categories $\mathsf{Mod}_{\Z[G]}\cong \mathsf{Ab}^G$ and $H\Z[G]\text{-}\mathsf{Mod}\cong (H\Z\text{-}\mathsf{Mod})^G$. There is an inclusion of categories $\mathsf{Mod}_{\Z[G]}\xhookrightarrow{}\mathsf{sMod}_{\Z[G]}$ sending a $\Z[G]$-module $M$ to the \textit{constant} simplicial $\Z[G]$-module on $M$, denoted by $\underline M=\underline{M}_\bullet$. By \cite[p. 332]{Schwede}, the Eilenberg--MacLane functor $H: \mathsf{Mod}_{\Z[G]}\to H\Z[G]\text{-}\mathsf{Mod}$ upgrades to a functor
$$
H: \mathsf{sMod}_{\Z[G]}\longrightarrow H\Z[G]\text{-}\mathsf{Mod}
$$
such that, for $M_\bullet\in \mathsf{sMod}_{\Z[G]}$, the underlying infinite loop space of $HM_\bullet$ is the realisation $|M_\bullet|$.

Given a simplicial $\Z[C_2]$-module $M_\bullet$, we will write $(M_\bullet)_{hC_2}$ for the simplicial abelian group
$$
(M_\bullet)_{hC_2}:= \mathrm{Diag}(M_\bullet \otimes_{\Z[C_2]} \Z[E_\bullet C_2]): [\hspace{1pt}p]\longmapsto M_p\otimes_{\Z[C_2]}\Z[C_2^{\times(\hspace{1pt} p+1)}].
$$
Now, given a $C_2$-spectrum $X$, we will write $X_{hC_2}$ for the spectrum $X\wedge_{C_2}(EC_2)_+$. The following result says that both meanings of $(-)_{hC_2}$ are intertwined by the Eilenberg--MacLane functor.

\begin{lem}\label{EMhC2Lemma}
    For $M_\bullet\in \mathsf{sMod}_{\Z[C_2]}$, there is a natural equivalence of spectra
    $$ (HM_\bullet)_{hC_2}\xrightarrow{\sim}H((M_\bullet)_{hC_2}).
    $$
\end{lem}

\begin{proof}
    According to \cite[Lem. 4.2]{Schwede}, there is a natural equivalence
    $$
HM_\bullet\wedge_{H\Z[C_2]} H(\Z[E_\bullet C_2])\xrightarrow{\sim} H((M_\bullet)_{hC_2}),
    $$
    where we are using that $H(\Z[E_\bullet C_2])$ is a cofibrant $H\Z[C_2]$-module. But for a simplicial set $X_\bullet$, $H(\Z[X_\bullet])$ is the free $H\Z$-module on $X_\bullet$, i.e., it is (equivalent to) $H\Z\wedge |X_\bullet|_+$. Thus
    \begin{align*}
        HM_\bullet\wedge_{H\Z[C_2]} H(\Z[E_\bullet C_2])&\cong \operatorname{colim}_{C_2} HM_\bullet\wedge_{H\Z} H(\Z[E_\bullet C_2])\\
        &\simeq \operatorname{colim}_{C_2} HM_\bullet\wedge_{H\Z} (H\Z\wedge (EC_2)_+)\\
        &\simeq HM_\bullet\wedge (EC_2)_+\\
        &=: (HM_\bullet)_{hC_2}.
    \end{align*}
    In the second and third equivalences, we are implicitly using that the $C_2$-spectra we are taking the colimit of are free, thus the colimit coincides with the homotopy colimit.
\end{proof}

Finally, we will denote $C:\mathsf{sAb}\to \mathsf{Ch}_{\geq 0}(\Z)$ for the functor sending a simplicial abelian group $(M_\bullet, \delta_\bullet)$ to the chain complex $(CM_\bullet, d_\bullet)$
$$
CM_n:= M_n, \quad d_n=\sum_{i=0}^{n}(-1)^{i}\delta_i: M_n\longrightarrow M_{n-1}.
$$
The normalised chain complex $NM_\bullet$ is a sub-complex of $CM_\bullet$, and in fact the inclusion $NM_\bullet\xhookrightarrow{}CM_\bullet$ is a split quasi-isomorphism \cite[$\S$III.2, Thm. 2.1 $\&$ Thm. 2.4]{GoerssJardine}. 

\begin{proof}[Proof of Theorem \ref{ActualThmC}]
Our goal is to identify $\Falg_\bullet(A)$ with $\underline{A}_{h{C_2}}$, functorially in $A\in \mathsf{Mod}_{\Z[C_2]}$. To do so, we first compare $N\underline{A}_{hC_2}$ and $\mathbb{A}_{hC_2}$.

Write $M_\bullet\overset{\epsilon}\twoheadrightarrow\Z$ for the minimal resolution of $\Z$ by free $\Z[C_2]$-modules
$$
\begin{tikzcd}
\dots\rar &\Z[C_2]\rar["1+t"] &\Z[C_2]\rar["1-t"]&\Z[C_2]\rar[two heads,"\epsilon"] &\Z,
\end{tikzcd}
$$
where $\epsilon$ sets $t=1$. The chain complex $C(\Z[E_\bullet C_2])$ together with the augmentation map $\epsilon: C(\Z[E_\bullet C_2])_0=\Z[C_2]\to \Z$ provides another such resolution (also known as the \textit{canonical resolution} of $\Z$ by free $\Z[C_2]$-modules). Therefore, there is a map $C(\Z[E_\bullet C_2])\to M_\bullet$ of resolutions of $\Z$ which, upon applying $A\otimes_{\Z[C_2]}(-)$, provides a quasi-isomorphism of chain complexes $C\underline{A}_{hC_2}\overset{\simeq}\longrightarrow \mathbb{A}_{hC_2}$. We thus obtain the desired quasi-isomorphism of chain complexes
$$
\begin{tikzcd}
\phi^{A}_\bullet: N\underline{A}_{hC_2}\rar[hook, "\simeq"] &C\underline{A}_{hC_2}\rar["\simeq"] & \mathbb{A}_{hC_2},
\end{tikzcd}
$$
which is clearly functorial in $A$.

The Dold--Kan correspondence (\ref{DoldKan}) can be upgraded to a Quillen equivalence of model categories \cite[$\S$II.4]{QuillenDoldKan} (for the projective model structure on chain complexes), so there is a zig-zag of equivalences (functorial in $A\in \mathsf{Mod}_{\Z[C_2]}$)
\begin{equation}\label{zigzagequivalenceFalgThmC}
\begin{tikzcd}
\Falg_\bullet(A)\rar["(\psi^{A}_\bullet)^{\vee}", "\sim"'] & \Gamma\mathbb{A}_{hC_2} &\lar["(\phi^{A}_\bullet)^\vee"', "\sim"] \underline{A}_{hC_2}.
\end{tikzcd}
\end{equation}
Each map in the zig-zag is an equivalence since $\Gamma$ preserves equivalences ($\Gamma$ is a right Quillen functor and every object in $\mathsf{Ch}_{\geq 0}(\Z)$ is fibrant) and because $N$ and $\Gamma$ are inverse to each other (as the Dold--Kan correspondence (\ref{DoldKan}) is an actual equivalence of categories). Applying geometric realisation to (\ref{zigzagequivalenceFalgThmC}) and noting Lemma \ref{EMhC2Lemma}, there results a zig-zag of infinite loop spaces
\begin{align*}
    |\Falg_\bullet(A)|\overset{\simeq}\longrightarrow |\Gamma \mathbb{A}_{hC_2}|&\overset{\simeq}\longleftarrow |\underline{A}_{hC_2}|= \Omega^\infty(H(\underline{A}_{hC_2}))\overset{\simeq}\longleftarrow \Omega^\infty((HA)_{hC_2}),
\end{align*}
which once again is functorial in $A$. This finishes the proof.
\end{proof}

\begin{rem}\label{technicalrem}
    A key step in the proof is the identification of $\Falg_\bullet(A)$ with $\Gamma\mathbb{A}_{hC_2}$ in~\eqref{zigzagequivalenceFalgThmC}, via the (homotopical) Dold--Kan correspondence. Crucially, this relies on the fact that $\Falg_\bullet(A)$ is a simplicial abelian group (cf. Lemma \ref{FalgIsSimplicial}), rather than just semi-simplicial, as we are unaware of a generalisation of Dold--Kan to the semi-simplicial setting. This hopefully clarifies the need for Proposition \ref{SubFacelem} and Lemma \ref{FalgIsSimplicial}.
    

\end{rem}

\subsection{Proof of Theorem \ref{ThmB}}\label{subsectionFdgeometry}
Let $F_\bullet(M)$ denote the simplicial homotopy fibre
\begin{equation}\label{geomFddefn}
F_\bullet(M):=\operatorname{holim}\left(\begin{tikzcd}
\{M^d\}\rar &\widetilde{\mathcal M}^h_\bullet&\lar[hook']\widetilde{\mathcal M}^s_\bullet
\end{tikzcd}\right).
\end{equation}
It has as $p$-simplices
$$
F_p(M)=\left\{W\in \widetilde{\mathcal{M}}^h_{p+1}: W_0=M, \quad \partial_0W\in \widetilde{\mathcal M}^s_p\right\},
$$
and as face maps
$$
\delta_i: F_p(M)\longrightarrow F_{p-1}(M), \quad W\longmapsto \partial_{i+1}W=W_{\langle 0,\dots,\widehat{i+1},\dots, p+1\rangle}, \quad i=0,\dots,p.
$$

Recall that $C_2$ acts on $\Wh(M)$ by $t\cdot \tau:=(-1)^{d-1}\overline\tau$. Our task is to find an equivalence
$
|F_\bullet(M)|\simeq \Omega^\infty(H\Wh(M)_{hC_2})
$, and by Theorem \ref{ActualThmC} we already know that the latter space is equivalent to $|\Falg_\bullet(\Wh(M))|$. Therefore, in order to show Theorem \ref{ThmB}, it suffices to establish an equivalence of semi-simplicial sets between $F_\bullet(M)$ and $\Falg_\bullet(\Wh(M))$.

Let $W^{d+p+1}\in F_{p}(M)$. We may find some face preserving homotopy equivalence $f: W\overset{\simeq_h}\longrightarrow_\Delta M\times \Delta^{p+1}$ with $f_0=\mathrm{Id}_M: W_0=M\to M$. With such a homotopy equivalence we can identify $\pi_1(W_K)$ with $\pi_1(M)$ for every sub-complex $K\in \subcomp_p$ so that for $K\subset L$, the identifications $\pi_1(W_K)\cong\pi_1(M)$ and $\pi_1(W_L)\cong \pi_1(M)$ are compatible with the induced isomorphism $\pi_1(W_K\xhookrightarrow{}W_L):\pi_1(W_K)\cong \pi_1(W_L)$. Any other choice of such a homotopy equivalence $f'$ will not change this identification, as $f$ and $f'$ are homotopic rel $W_0=M$ (to see this, note that it is equivalent to showing that a homotopy equivalence $f: M\times\Delta^p\overset{\simeq}\longrightarrow M\times \Delta^p$ with $f_0=\mathrm{Id}_M$ is homotopic rel $M\times \Delta^p_0$ to the identity. This in part follows from the fact that it is block homotopic rel $M\times \Delta^p_0$ to the identity by an \textit{Alexander trick}-like argument similar to \cite[Lem. 2.1]{BurgLashRoth} and that $h\mathrm{Aut}(M)_\bullet\simeq \widetilde{h\mathrm{Aut}}(M)_\bullet$). Once we have made such an identification of fundamental groups, we can define the functor $\tau_W\in \mathrm{Fun}(\subcomp_p, {\Wh(M)})$
$$
\tau_W(L,K):=\tau(W_K\lhook\joinrel\xrightarrow{\ \simeq\ }W_L)\in \Wh(M).
$$
The composition rule (\ref{whiteheadtorsioncomposition}) of the Whitehead torsion and the inclusion-exclusion principle (\ref{whiteheadtorsincexc}) guarantees that $\tau_W$ satisfies (\ref{incexcprinciple}) and therefore, by Lemma \ref{incexclemma}, $\tau_W$ is an element of $\mathrm{Fun}^\square(\subcomp_p,{\Wh(M)})$. In fact, by the definition of the $F_\bullet(M)$, the functor $\tau_W$ is a $p$-simplex in this space (remember that $\tau^*$ in (\ref{generalisedfacehornduality}) should now be replaced by $(-1)^{d-1}\overline{\tau}$).

\begin{prop}\label{Fdgeomvsalg}
For $d=\dim M\geq 5$, there is an equivalence of semi-simplicial sets
$$
\tau_{(-)}: F_\bullet(M)\overset{\simeq}\longrightarrow \Falg_\bullet(\Wh(M)), \quad W\longmapsto \tau_W.
$$
\end{prop}
Together with Theorem \ref{ActualThmC}, this will prove Theorem \ref{ThmB}. In particular, by Proposition \ref{quasiisoprop},
\begin{equation}\label{FMhtpygroups}
    \pi_n(F_\bullet(M))\cong H_n(C_2;\Wh(M))\cong\left\{
    \begin{array}{cc}
        \frac{\Wh(M)}{\left\{b+ (-1)^d\overline{b}\ \mid \ b\in \Wh(M)\right\}}, & n=0, \\[10pt]
    \frac{\left\{a\in \Wh(M)\ \mid \ a=(-1)^{d+n} \overline{a}\right\}}{\left\{b+(-1)^{d+n} \overline{b}\ \mid\ b\in \Wh(M)\right\}}, & n\geq 1. 
    \end{array}
    \right.
\end{equation}
\begin{proof}[Proof of Proposition \ref{Fdgeomvsalg}]
We need to prove that $\tau_{(-)}$ induces isomorphisms in homotopy groups. To prove surjectivity, let $a\in \Wh(M)$ be such that $a=(-1)^{d+n}\overline a$ and let $W^{d+n+1}: M\times \Lambda_0^{n+1}\hcob W_{\langle 1,\dots,n+1\rangle}$ be an $h$-cobordism rel boundary with
$
\tau(W, M\times \Lambda_0^{n+1})=a
$. The manifold $W$ admits a stratified structure over $\Delta^{n+1}$ with $0$-th horn $\Lambda_0(W)=M\times \Lambda_0^{n+1}$ and $0$-th face $W_{\langle1,\dots,n+1\rangle}$. It is not difficult to see that $\tau_W=\tau_{a}$ by Proposition \ref{SubFacelem} and noting that $\tau_W$ satisfies face-horn dualities for all faces (in particular $\Delta^{n+1}$).

For injectivity, let $W\in F_n(M)$ be a cycle such that $[\tau_W]=0$ in $\pi_n(\Falg_\bullet(\Wh(M))$. By the first step in the proof of Proposition \ref{quasiisoprop}, $\tau_W=\tau_{b+(-1)^{d+n}\overline b}$ for some $b\in \Wh(M)$. Let $V: W\hcob W'$ be an $h$-cobordism rel boundary with torsion $\tau(V,W)=-b$ (after having identified $\pi_1(W)$ with $\pi_1(M)$ appropriately. Then $\pi_1(V)$ and $\pi_1(W')$ get identified with $\pi_1(M)$ too). We claim $W'$ is (face-preservingly) diffeomorphic to $M\times\Delta^{n+1}$, i.e., we have to show that $\tau(W',\Lambda_0(W'))=0$ by the $s$-cobordism theorem. Since $\Lambda_0(W')= \Lambda_0(W)=M\times\Lambda_0^{n+1}$,
$$
\tau(W,\Lambda_0(W))+\tau(V,W)=\tau(W',\Lambda_0(W'))+\tau(V,W')
$$
which, by duality, yields
$$
\tau(W',\Lambda_0(W'))=b+(-1)^{d+n}\overline b-b-(-1)^{d+n+1}\overline{(-b)}=0.
$$
Let $\phi: W'\cong M\times \Delta^{n+1}$ be a diffeomorphism fixing $\Lambda_0(W')=M\times \Lambda_0^{n+1}$, and consider $V':=M_\phi\circ V$, where $M_\phi: W'\scob M\times \Delta^{n+1}$ denotes the \textit{mapping cylinder}\footnote{This should really be the mapping cylinder \textit{with collars} (see Definition \ref{stratifieddefn}($iii$)). Namely, given a diffeomorphism (possibly rel boundary) $\phi: A\cong B$, we define $M_\phi$ by $A\times [0,1/2]\cup_{\phi\times\{1/2\}}B\times[1/2,1]$.} of $\phi$. Then using the canonical diffeomorphism rel boundary $\Delta^{n+1}\cong \Lambda_1^{n+2}$, the manifold $V'$ admits a stratified structure over $\Delta^{n+2}$ with $\partial_1V'=W$ and $\Lambda_1(V')=M\times \Lambda_1^{n+2}$. Therefore $V'$ provides a null-homotopy of $W$ in $F_\bullet(M)$, as desired. 
\end{proof}
\begin{rem}\label{degeneraciesBdiffh}
The semi-simplicial sets $\widetilde{\mathcal M}^h_\bullet$ and $\widetilde{\mathcal M}^s_\bullet$ admit compatible systems of degeneracies that make them into simplicial objects. Namely, the $i$-th degeneracy map $s_i: \widetilde{\mathcal M}^{h/s}_p\to \widetilde{\mathcal M}^{h/s}_{p+1}$ sends a $p$-simplex $W^{d+p}\Rightarrow \Delta^p$ to the pullback $W \ \leftidx{_\mathrm{pr}}{\!\times}_{s^{i}}\Delta^{p+1}$, where $\mathrm{pr}: W\to \Delta^p$ is the composition $W\subset \R^\infty\times \Delta^p\twoheadrightarrow \Delta^p$, and $s^{i}: \Delta^{p+1}\to \Delta^p$ is the linear $i$-th codegeneracy map. The pullback $W \ \leftidx{_\mathrm{pr}}{\!\times}_{s^{i}}\Delta^{p+1}$ is regarded as a manifold stratified over $\Delta^{p+1}$ under the inclusion
$$
W \ \leftidx{_\mathrm{pr}}{\!\times}_{s^{i}}\Delta^{p+1}\xhookrightarrow{}\R^\infty\times \Delta^{p+1}, \quad ((w,x),y)\longmapsto (w,y),
$$
for $(w,x)\in W\subset \R^\infty\times \Delta^p$ and $y\in \Delta^{p+1}$ such that $x=s^{i}(y)$. The semi-simplicial homotopy fibre $F_\bullet(M)$ thus inherits a simplicial structure which agrees with that of $\Falg_\bullet(\Wh(M))$, i.e., the map $\tau_{(-)}$ of Proposition \ref{Fdgeomvsalg} becomes an equivalence of simplicial sets.
\end{rem}

\subsection{Relation to the Rothenberg exact sequence}\label{Rothsection}
The purpose of this section is to derive a consequence of Theorem \ref{ThmB} in a different direction to Theorem \ref{diffdifference}. The reader may want to skip it on first reading.

For a finite group $G$ and a naïve $G$-spectrum $X$, we will denote  by $X^{tG}$ the \textit{Tate construction} of $X$ \cite[Defn. 2.2]{TateHomology}, i.e., the homotopy cofibre of the \textit{norm map} (cf. \cite[Prop. 2.4]{WWII}) $N: X_{hG}\to X^{hG}$, where $X_{hG}:= X\wedge_G (EG)_+$ are the homotopy $G$-orbits of $X$ as before, and $X^{hG}:=F(\Sigma_+^\infty EG, X)^G$ denotes the homotopy $G$-fixed points of $X$. Here $F(-,-)^G$ is the $G$-equivariant mapping spectrum. When $X=HA$ for some $\Z[G]$-module $A$, $\pi_*^s((HA)_{hG})=H_*(G;A)$ whilst $\pi_*^s((HA)^{hG})=H^{-*}(G;A)$, and the norm map in degree zero $A_G\to A^G$ is multiplication by the norm element $N=\sum_{g\in G}g\in \Z[G]$. Therefore when $G=C_2$ and $A=\Wh(M)$,
$$
\widehat{H}^*(C_2;\Wh(M)):=\pi_{*}^s(H\Wh(M)^{tC_2})=\left\{
\begin{array}{cc}
   H_{*-1}(C_2;\Wh(M)),  & *\geq 2, \\[10pt]
    \tfrac{\left\{a\in \Wh(M)\ \mid \ a=(-1)^{d} \overline{a}\right\}}{\left\{b+(-1)^{d} \overline{b}\ \mid\ b\in \Wh(M)\right\}}\subset H_0(C_2;\Wh(M)), & *=1,\\[10pt]
    \tfrac{\left\{a\in \Wh(M)\ \mid \ a=(-1)^{d-1} \overline{a}\right\}}{\left\{b+(-1)^{d-1} \overline{b}\ \mid\ b\in \Wh(M)\right\}}\subset H^0(C_2;\Wh(M)), & *=0,\\[10pt]
    H^{-*}(C_2;\Wh(M)), & *\leq -1.
    
\end{array}
\right.
$$

Let $\mathbb{L}^{h/s}(M)$ denote the quadratic ordinary/simple \textit{$L$-theory spectrum} of $M^d$ \cite[$\S$13]{RanickiAlgSurg}, whose homotopy groups are the $L$-theory groups $L^{h/s}_*(\Z[\pi_1 M])$. In this section we establish a spacified version of the positive-degree part of the \textit{Rothenberg exact sequence} for quadratic $L$-theory \cite[Prop. 1.10.1]{Ranickiexactseq}.

\begin{prop}\label{Rothprop}
For $d\geq 5$, there is an equivalence of spaces 
\begin{equation}\label{RothCorequation}
\Omega^{\infty+d+1}\operatorname{hofib}\left(\mathbb{L}^s(M)\to \mathbb{L}^h(M)\right)\simeq \Omega^{\infty+1}(H\Wh(M)^{tC_2}).
\end{equation}
\end{prop}

\begin{proof}Let $\widetilde{\mathcal S}^{h/s}_\bullet(M)$ denote the ordinary/simple \textit{block structure spaces} of $M$ \cite{QuinnSSP}. Roughly speaking, a $p$-simplex in the space $\widetilde{\mathcal S}^{h/s}_\bullet(M)$ is a pair $(W^{d+p},f)$ consisting of a manifold $W$ stratified over $\Delta^p$ and a face-preserving homotopy equivalence $f: W\overset{\simeq_{h/s}}\longrightarrow_{\Delta}M\times\Delta^p$. For $d\geq 5$, surgery theory establishes a diagram of fibration sequences \cite[$\S$3]{QuinnSSP}
$$
\begin{tikzcd}
\Omega^{\infty+d+1}\mathbb{L}^s(M)\dar\rar&\widetilde{\mathcal S}^s(M)\dar\rar&(G/O)^{M_+}_*\dar[equal]\\
\Omega^{\infty+d+1}\mathbb{L}^h(M)\rar&\widetilde{\mathcal S}^h(M)\rar&(G/O)^{M_+}_*.
\end{tikzcd}
$$
Taking homotopy fibres we obtain an equivalence
\begin{equation}\label{LtheorySurgeryequivalence}
\Omega^{\infty+d+1}\operatorname{hofib}\left(\mathbb{L}^s(M)\to \mathbb{L}^h(M)\right)\simeq\operatorname{hofib}\left(\widetilde{\mathcal S}^s(M)\longrightarrow\widetilde{\mathcal S}^h(M)\right).
\end{equation}
On the other hand, it is not difficult to see that there is another diagram of fibration sequences
$$
\begin{tikzcd}
s\mathrm{Aut}(M)\dar[hook, "\substack{\text{incl. of}\\ \text{cpts.}}"]\rar&\widetilde{\mathcal S}^s(M)\dar\rar["u"]&\widetilde{\mathcal M}^s\dar\\
h\mathrm{Aut}(M)\rar&\widetilde{\mathcal S}^h(M)\rar["u"] &\widetilde{\mathcal M}^h,
\end{tikzcd}
$$
where $u$ is the (geometric realisation of the) forgetful map sending a $p$-simplex $(W^{d+p},f)$ in $\widetilde{\mathcal S}^{h/s}_\bullet(M)$ to $W\in \widetilde{\mathcal M}^{h/s}_p$. Taking again homotopy fibres (in the basepoint components corresponding to $M$ and $\mathrm{Id}_M$), we get a map $\overline{u}:\operatorname{hofib}(\widetilde{\mathcal S}^s(X)\longrightarrow\widetilde{\mathcal S}^h(X))\to |F_\bullet(M)|\simeq \Omega^\infty(H\Wh(M)_{hC_2})$ which is an equivalence onto the components that are hit. For each $[a]\in H_0(C_2;\Wh(M))$, write $\Omega^\infty_{[a]}(H\Wh(M)_{hC_2})\subset \Omega^\infty(H\Wh(M)_{hC_2})$ for the connected component corresponding to $[a]$. There is hence a chain of equivalences
\begin{align*}
\Omega^{\infty+d+1}\operatorname{hofib}\left(\mathbb{L}^s(M)\to \mathbb{L}^h(M)\right)\overset{(\ref{LtheorySurgeryequivalence})}\simeq \operatorname{hofib}\left(\widetilde{\mathcal S}^s(X)\to\widetilde{\mathcal S}^h(X)\right)\simeq\hspace{-3pt} \bigsqcup_{[a]\in \mathrm{Im} (\pi_0(\overline{u}))} \Omega^{\infty}_{[a]}(H\Wh(M)_{hC_2}).
\end{align*}

To establish (\ref{RothCorequation}), it remains to argue that $\mathrm{Im}(\pi_0(\overline{u}))=\widehat{H}^1(C_2;\Wh(M))$. Choose a $0$-simplex in $\operatorname{hofib}(\widetilde{\mathcal S}^s_\bullet(X)\longrightarrow\widetilde{\mathcal S}^h_\bullet(X))$, that is, a $1$-simplex $(W,f)\in \widetilde{\mathcal S}^h_1(M)$ such that $W_0=M$, $f_0=\mathrm{Id}_M$ and $f_1: W_1\overset{\simeq_s}\longrightarrow M\times\{1\}$ is a simple homotopy equivalence. In particular, $W: M\hcob W_1$ is an $h$-cobordism starting at $M$. By definition of $\psi_\bullet^{\Wh(M)}$ (see (\ref{psiAdefn})), 
$$
\pi_0(\overline u): \pi_0\left(\operatorname{hofib}\left(\widetilde{\mathcal S}^s(X)\longrightarrow\widetilde{\mathcal S}^h(X)\right)\right)\longrightarrow H_0(C_2,\Wh(M)), \quad [W,f]\longmapsto [\tau(W,M)].
$$
Now $0=\tau(\hspace{2pt}f_0)=(i_M)_*^{-1}\tau(\hspace{2pt}f)+\tau(W,M)$ and, by duality, $0=(h^W_*)^{-1}\tau(\hspace{2pt}f_1)=(i_M)_*^{-1}\tau(\hspace{2pt}f)+(-1)^{d}\overline{\tau}(W,M)$. Putting these two together we obtain $\tau(W,M)=(-1)^{d}\overline{\tau}(W,M)$, so it follows that $\mathrm{Im}(\pi_0(\overline u))\subset \widehat{H}^1(C_2;\Wh(M))$. For the other inclusion, let $W\in F_0(M)$ be such that $\tau(W,M)=(-1)^{d}\overline{\tau}(W,M)$, and pick some face-preserving homotopy equivalence $f: W\overset{\simeq_h}\longrightarrow_\Delta M\times I$. By post-composing $f$ with $f_0^{-1}\times I$, for some $f_0^{-1}$ homotopy inverse to $f_0$, we may assume that $f_0=\mathrm{Id}_M$. Then
$$
\tau(\hspace{2pt}f_1)=(i_{W_1})_*^{-1}\tau(\hspace{2pt}f)+(-1)^{d}h^W_*\overline{\tau}(W,M)=(i_{W_1})_*^{-1}\tau(\hspace{2pt}f)+h^W_*\tau(W,M)=h^W_*\tau(\hspace{2pt}f_0)=0,
$$
so $f_1: W_1\simeq_s M\times \{1\}$ is a simple homotopy equivalence and thus $(W,f)$ represents a $0$-simplex in $\operatorname{hofib}(\widetilde{\mathcal S}^s_\bullet(X)\longrightarrow\widetilde{\mathcal S}^h_\bullet(X))$. This concludes the proof of Propostion \ref{Rothprop}.
\end{proof}

\begin{rem}[Speculative]

The equivalence (\ref{RothCorequation}) can presumably be upgraded to one of infinite loop spaces, and should hold also for $d<5$ (we need this assumption in our statement only because of the use of surgery theory in the proof). The argument should be completely "surgery-free" but still similar to that of Theorem \ref{ThmB}, by replacing the block moduli spaces $\widetilde{\mathcal M}^{h/s}_\bullet$ with the $L$-theory semi-simplicial sets $\mathsf{L}^{h/s}_\bullet$ as defined in \cite[$\S$2]{QuinnSSP}. More generally, an equivalence of spectra $\Sigma^{-d}\operatorname{hofib}\left(\mathbb{L}^s(M)\to \mathbb{L}^h(M)\right)\simeq H\Wh(M)^{tC_2}$ should hold.
\end{rem}

\begin{rem}[Relation to Weiss--Williams II]
    As will be explained in Appendix \ref{appendixB}, Theorem \ref{ThmB} is closely related to the first part of the seminal series \textit{Automorphisms of Manifolds and Algebraic $K$-Theory} by Weiss and Williams \cite{WWI}. That said, the techniques we have used in this section are reminiscent of those in Part II of the same series \cite{WWII}, as was pointed out to us by the anonymous referee, to whom we are grateful. Let us briefly comment on this now.

    Given a ring with anti-involution $R$ (e.g., $R = \Z[\pi_1 M]$ with the anti-involution described in~(\ref{Zpi1antiinvolution})), the authors of~\cite{WWII} construct a map from the $L$-theory spectrum of $R$ (with various decorations, such as $h$ and $s$) to $\mathrm{K}(R)^{tC_2}$, the Tate construction of the algebraic $K$-theory spectrum of $R$ (with the corresponding decoration). As noted in~\cite[p.~52]{WWII}, these maps can be used to derive ``higher order'' Rothenberg exact sequences, which are not quite the one we study in Proposition~\ref{Rothprop}---our version concerns the ``change of decoration'' from~$h$ to~$s$, whereas theirs arise from the Postnikov truncations~$\tau_{\geq *}\mathrm{K}(R)$.

To construct these maps, they present a simplicial resolution of $\mathrm{K}(R)$ given by
\[
[n] \mapsto \mathrm{K}(\Fun(\face_n, \mathsf{Ch}_R)),
\]
where $\mathsf{Ch}_R$ denotes the Waldhausen category of chain complexes of projective left $R$-modules. They then introduce a notion of \emph{duality} in $\mathrm{K}(\Fun(\face_n, \mathsf{Ch}_R))$ of Poincaré duality flavour~\cite[pp.~77--78]{WWII}, which is reminiscent of our \emph{face-horn} duality~\eqref{facehornduality}. By similar techniques to those of Vogell~\cite{VogellInvolution}, this duality upgrades $\mathrm{K}(\Fun(\face_\bullet, \mathsf{Ch}_R))$ to a simplicial $C_2$-spectrum.

By contrast, our work concerns only the (reduced) first algebraic $K$-theory group of $R$, rather than the entire spectrum. For instance, given an object $C \in \Fun(\face_n, \mathsf{Ch}_R)$ such that $C(\sigma)$ is free and finitely generated for all $\sigma \subset \Delta^n$ and such that $C(\sigma) \to C(\xi)$ is a quasi-isomorphism for all $\sigma \subset \xi \subset \Delta^n$, we focus on the functor $\face_n \to \widetilde{K}_1(R)$ that assigns to each face inclusion $\sigma \subset \xi$ the torsion of $C(\sigma) \to C(\xi)$. On a different note, our $\Falg(-)$-construction models the homotopy $C_2$-orbit spectrum $H(-)_{hC_2}$, as opposed to the $C_2$-spectrum $H(-)$.

In any case, even though our work and that in~\cite{WWII} are clearly related, we have not yet fully worked out the details of how they connect, but it would be very interesting to do so.

\end{rem}

\section{\texorpdfstring{Proof of Theorem \ref{diffdifference}($i$)}{Proof of Theorem A(i)}}\label{section4}
By analysing the lower-degree portion of the long exact sequence of homotopy groups associated to the homotopy pullback of Theorem \ref{ThmB}, we propose a general strategy to prove Theorem \ref{diffdifference}$(i)$ (see Proposition \ref{conditionprop}). We then present an example of an $h$-cobordism $W: L\hcob M$, where $L$ is as in the statement of Theorem \ref{diffdifference}, which satisfies the conditions of the proposed strategy. All throughout let $M^d$ denote a closed smooth manifold of dimension $d\geq 5$.

\subsection{A general strategy}
From the homotopy cartesian square of the homotopy fibre $F_\bullet(M)$ (see (\ref{geomFddefn})) we obtain an associated long exact sequence of homotopy groups
\begin{equation}\label{lesMhMsF}
\begin{tikzcd}[column sep = 13 pt, row sep=9 pt]
\dots \rar&\pi_n(F_\bullet(M))\rar &\pi_n(\widetilde{\mathcal M}^s,\{M\})\rar &\pi_n(\widetilde{\mathcal M}^h,\{M\})\rar&\dots& \\
&\dots\rar &\pi_1(\widetilde{\mathcal M}^h,\{M\})\rar["\partial"] &\pi_0(F_\bullet(M))\rar &\pi_0(\widetilde{\mathcal M}^s)\rar &\pi_0(\widetilde{\mathcal{M}}^h).
\end{tikzcd}
\end{equation}
For $n\geq 1$, the boundary map $\partial: \pi_n(\widetilde{\mathcal M}^h,\{M\})\to \pi_{n-1}(F_\bullet(M))$ sends an $n$-cycle $W^{d+n}\in\widetilde{\mathcal M}^h_n$ based at $M$ to $W$ as an $(n-1)$-cycle in $F_{\bullet}(M)$. So the image of the lowest-degree boundary map $\partial: \pi_1(\widetilde{\mathcal M}^h,\{M\})\to \pi_0(F_\bullet(M))$ consists of those classes represented by $h$-cobordisms $W: M\hcob M$. 

\begin{defn}
An $h$-cobordism $W: M\hcob M'$ is said to be \textbf{inertial} \cite[Defn. 2.1]{JahrenKwasikInertiaHcobs} if $M'$ is diffeomorphic to $M$. The set of inertial $h$-cobordisms starting at $M$ (up to diffeomorphism rel $M$) is denoted by $I(M)\subset h\mathrm{Cob}(M)\cong\Wh(M)$.
\end{defn}

\begin{exm}
Given an $h$-cobordism $W: M\hcob M'$, denote $\overline W:M'\hcob M$ for $W$ with the reversed cobordism direction. The \textit{double} $D(W):=\overline W\circ W=W\cup_{M'}\overline{W}$ \cite[p. 400]{MilnorWhiteheadTorsion} is an inertial $h$-cobordism $M\hcob M$ with torsion (see (\ref{hcobdualformula}) and (\ref{torsofhcobcomposition}))
$$
\tau(D(W), M)=\tau(W,M)+(-1)^{d}\overline{\tau}(W,M).
$$
The subgroup of \textit{double} $h$-cobordisms of $M^d$,
\begin{equation}\label{doublesubgroup}
\mathcal D(M):=\{\sigma+(-1)^d\overline\sigma: \sigma\in \Wh(M)\}\subset \Wh(M),
\end{equation}
is therefore a subset of $I(M)$ too. Also observe from (\ref{FMhtpygroups}) that 
$$
H_0(C_2;\Wh(M))=\Wh(M)/\mathcal D(M).
$$
\end{exm}

\begin{lem}
Let $\frac{I(M)}{\mathcal D(M)}$ denote the image of $I(M)$ under the projection $\Wh(M)\twoheadrightarrow{} H_0(C_2;\Wh(M))$. Under the isomorphism $\pi_0(F_\bullet(M))\cong H_0(C_2;\Wh(M))$ established in (\ref{FMhtpygroups}),
$$
\im\left\{\partial: \pi_1(\widetilde{\mathcal M}^h,\{M\})\to \pi_0(F_\bullet(M))\cong\frac{\Wh(M)}{\mathcal D(M)}\right\}=\frac{I(M)}{\mathcal D(M)}.
$$
\end{lem}
\begin{proof}
The inclusion $(\subset)$ is immediate. Conversely if $W: M\leadsto M'$ is an inertial $h$-cobordism with $\phi: M\cong M'$, let $W': M\leadsto M$ denote the $h$-cobordism $M_{\phi^{-1}}\circ W$. Recall that the isomorphism $\pi_0(F_\bullet(M))\cong H_0(C_2;\Wh(M))$ sends the class represented by $W$ to that of its torsion $\tau(W,M)=\tau(W',M)$. As $W'$ represents a class in $\pi_1(\widetilde{\mathcal M}^h,\{M\})$, we are done.
\end{proof}

Recall from Proposition \ref{connectedcompsofMs} and (\ref{connectedcompsofMh}) that $B\bdiff(M)$ and $B\bdiffh(M)$ are the connected components of $\widetilde{\mathcal M}^s$ and $\widetilde{\mathcal M}^h$, respectively, which contain $M^d$ as basepoint. We will denote $\bdiffh/\bdiff(M)\subset F_\bullet(M)$ for the union of connected components corresponding to $\frac{I(M)}{\mathcal D(M)}$. By exactness of (\ref{lesMhMsF}), these are exactly the components of $F_\bullet(M)$ that map to $B\bdiff(M)\subset \widetilde{\mathcal M}^s$. We thus obtain a fibration sequence
\begin{equation}\label{bdiffhmodbdifffibrationseq}
\begin{tikzcd}
\bdiffh/\bdiff(M)\rar & B\bdiff(M)\rar & B\bdiffh(M).
\end{tikzcd}    
\end{equation}

For the remaining of the section, let $W^{d+1}: L^d\hcob M^d$ be some $h$-cobordism with torsion $\tau:=\tau(W,L)\in \Wh(L)$. The homotopy long exact sequences of the fibration (\ref{bdiffhmodbdifffibrationseq}) for $L$ and $M$ yield the diagram
\begin{equation}\label{bdiffhextensionLM}
\begin{tikzcd}[column sep= 11 pt]
\pi_2(B\bdiffh(L))\drar[phantom, "(\dagger)"]\ar[d, "\cong", "\substack{\text{base pt.}\\ \mathrm{change}}"']\rar["\partial"]&H_1(C_2;\Wh(L))\ar[d,"\cong", "h^W_*"']\rar& \pi_1(B\bdiff(L))\rar &\pi_{1}(B\bdiffh(L))\ar[d, "\cong", "\substack{\text{base pt.}\\ \mathrm{change}}"']\ar[r, two heads, "\partial"] &\frac{I(L)}{\mathcal D(L)}\\
\pi_2(B\bdiffh(M))\rar["\partial"]&H_1(C_2;\Wh(M))\rar& \pi_1(B\bdiff(M))\rar &\pi_{1}(B\bdiffh(M))\ar[r, two heads, "\partial"] &\frac{I(M)}{\mathcal D(M)},
\end{tikzcd}
\end{equation}
where we have used the isomorphism $\pi_n(\bdiffh/\bdiff(-))\cong H_n(C_2;\Wh(-))$ for $n\geq 1$ from Proposition \ref{Fdgeomvsalg}. We are trying to compare the middle terms of the two extensions above, since
$$
\pi_1(B\bdiff(-))\cong\pi_0(\bdiff(-))=:\widetilde{\Gamma}(-).
$$
We first study the left part of the extensions in (\ref{bdiffhextensionLM}). 
\begin{prop}\label{leftextensionlemma}
For $n\geq 2$, the following square commutes:
$$
\begin{tikzcd}
\pi_n(B\bdiffh(L))\ar[d, "\cong", "\substack{\text{\normalfont{base pt.}}\\ \mathrm{change}}"']\rar["\partial"] & H_{n-1}(C_2; \Wh(L))\ar[d,"\cong", "h^W_*"']\\
\pi_n(B\bdiffh(M))\rar["\partial"] &H_{n-1}(C_2; \Wh(M)).
\end{tikzcd}
$$
In particular, the square decorated by $(\dagger)$ in (\ref{bdiffhextensionLM}) commutes.
\end{prop}

\begin{proof}
The basepoint change map sends an $n$-cycle $V^{d+n}\in B\bdiffh(L)_n$ to the manifold (see Figure \ref{Basepoint Change Figure})
$$
W_\#V:=V\cup_{L\times \partial\Delta^n} (W\times \partial \Delta^n).
$$
The union is made along the boundary $\partial V=L\times \partial\Delta^n$. The manifold $W_\#V$ is naturally stratified over $\Delta^n$, and clearly represents an $n$-cycle in $B\bdiffh(M)$. As mentioned before, the boundary map $\partial: \pi_n(B\bdiffh(L))\longrightarrow H_{n-1}(C_2;\Wh(L))$ sends $[V]$ to the class represented by $\tau(V,V_0)=\tau(V,L)$ in $H_{n-1}(C_2;\Wh(L))$. We thus need to show that
$$
\tau(W_\#V,M)\equiv h_*^W\tau(V,L)\mod \{\sigma+(-1)^{d+n-1}\overline \sigma: \sigma\in \Wh(L)\}.
$$
We compute $\tau(W_\#V,M)$ directly. For any subspaces $A, B\subset W_\#V$ with $A\subset B$, write $i_A^B$ for the inclusion. If $P:=V\cup_{V_0=L}W$ (see Figure \ref{Basepoint Change Figure}), we can factor the inclusion $i_M^{W_\#V}:M=M\times\{0\}\xhookrightarrow{}W_\#V$ as
$$
\begin{tikzcd}
M\dar[hook, "\simeq"]\ar[rr,hook, "\simeq"]&&W_\#V\\
W\ar[rr,hook, "\simeq"]&&P\ar[u,hook, "\simeq"].
\end{tikzcd}
$$
We compute the torsion of these three maps using the inclusion-exclusion principle (\ref{whiteheadtorsincexc}):
\begin{align*}
    \tau(W,M)&=(-1)^{d}h^W_*\overline\tau,\\
    \tau(P,W)&=(i_L^W)_*\tau(V,L)+(i_W^W)_*\tau(W,W)-(i_L^W)_*\tau(L,L)\\
    &=(i_L^W)_*\tau(V,L),\\
    \tau(W_\#V,P)&=(i_V^P)_*\tau(W_\#V,V)+(i_W^P)_*\tau(W,W)-(i_L^P)_*\tau(W,L)\\
    &=(i^P_V)_*(i_{\partial V}^{V})_*\tau(W\times\partial\Delta^n,L\times\partial\Delta^n)-(i_L^P)_*\tau\\
    &=\chi(\partial\Delta^n)\cdot (i^P_{L})_*\tau-(i_L^P)_*\tau\\
    &=(-1)^{n-1}(i^P_{L})_*\tau.
\end{align*}
In the penultimate line we have used that $i^P_V\circ i^{V}_{\partial V}\circ i_{L\times 0}^{L\times\partial\Delta^n}=i_L^P$ and the product rule (\ref{eulerchartors}) of $\tau(-)$, for which we need the condition $n\geq 2$ for $\partial\Delta^n$ to be connected. By the composition rule (\ref{whiteheadtorsioncomposition}), we get
\begin{align*}
    \tau(W_\#V,M)&=(-1)^dh^W_*\overline\tau+(i_M^W)^{-1}_*(i_L^W)_*\tau(V,L)+(-1)^{n-1}(i_M^P)^{-1}_*(i^P_{L})_*\tau\\
    &=(-1)^dh^W_*\overline\tau+h_*^W\tau(V,L)+(-1)^{n-1}h^W_*\tau\\
    &=h_*^W\tau(V,L)+(-1)^{n-1}\left(h^W_*\tau+(-1)^{d+n-1}\overline{h^W_*\tau}\right),
\end{align*}
where in the second line we have used the commutative diagram
$$
\begin{tikzcd}
M\rar[hook,"\simeq"]\ar[dr, hook, "\simeq"] &W\dar[hook, "\simeq"]&L\ar[l, hook', "\simeq"']\ar[ld, hook', "\simeq"]\\
&P,&
\end{tikzcd}
$$
so $(i_M^P)^{-1}_*(i^P_{L})_*=(i_M^W)^{-1}_*(i^W_{L})_*=h^W_*$. This finishes the proof.
\end{proof}
\begin{figure}[ht]
    \centering
    \includegraphics[scale = 0.1]{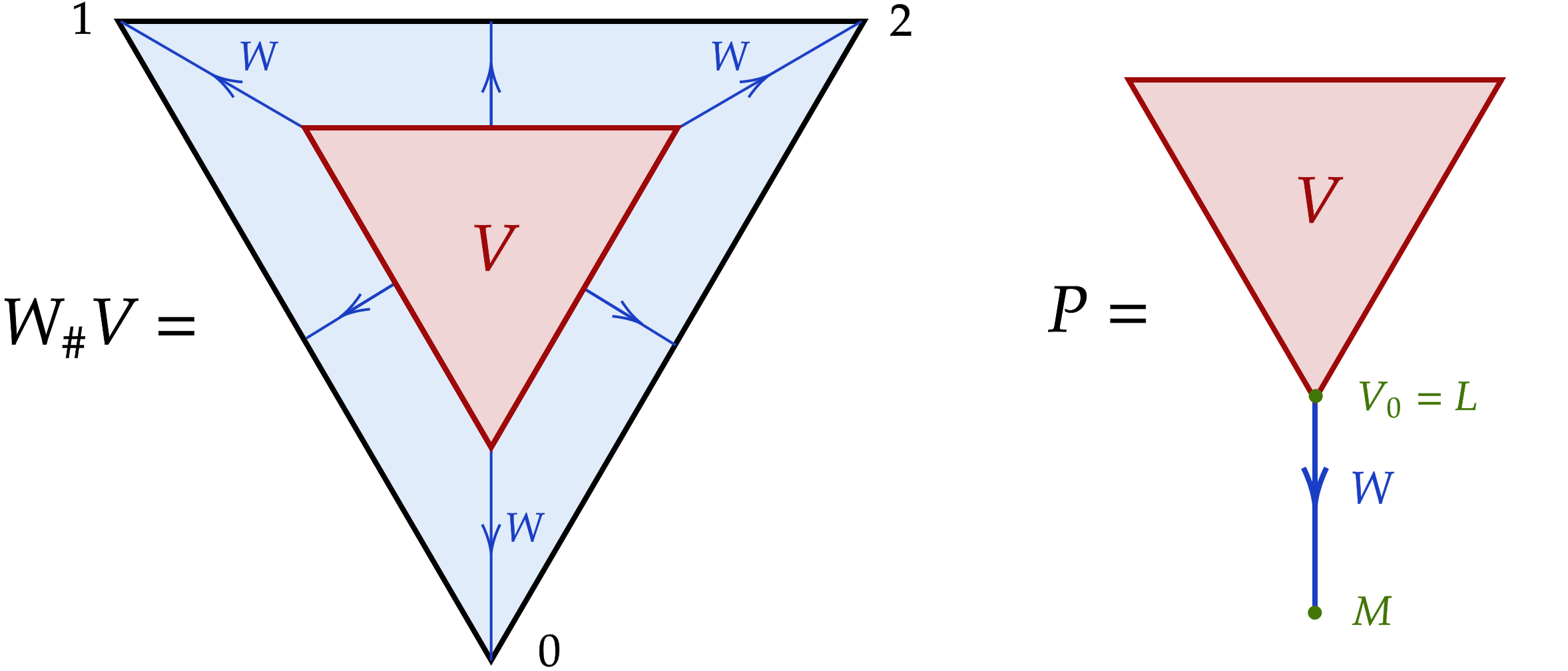}
    \caption{Illustration of $W_\#V$ and $P:=V\cup_{L\times\{0\}}W$ when $n=2$.}
    \label{Basepoint Change Figure}
\end{figure}

In the next section we will focus on the task of finding an example of $W:L\hcob M$ for which $\widetilde{\Gamma}(L)\neq \widetilde{\Gamma}(M)$ as in Theorem \ref{diffdifference}($i$). By the previous result, we should make the right hand sides of the two extensions in (\ref{bdiffhextensionLM}) differ. In order to do so, we will use the following.

\begin{prop}\label{conditionprop}
Let $W:L\hcob M$ be such that

\hspace{50pt}\begin{enumerate*}[label=\protect\circled{\normalfont{\Roman*}}]
    \item\label{condI} $\frac{I(L)}{\mathcal D(L)}=0$ but $\frac{I(M)}{\mathcal D(M)}\neq 0$,\hspace{30pt}
    \item\label{condII} $\pi_1(B\bdiff(L))$ is finite.
\end{enumerate*}

\noindent Then $\frac{I(M)}{\mathcal D(M)}$ is finite of cardinality $N>1$ and
$$
|\pi_1(B\bdiff(L))|=N\cdot|\pi_1(B\bdiff(M))|<\infty.
$$
In particular, $\widetilde{\Gamma}(L)\neq \widetilde{\Gamma}(M)$.
\end{prop}

\begin{proof}
If $\frac{I(L)}{\mathcal D(L)}=0$ then $\pi_1(B\bdiff(L))$ surjects onto $\pi_1(B\bdiffh(L))$, and hence the latter is finite too. As $\pi_1(B\bdiffh(M))\cong \pi_1(B\bdiffh(L))$ surjects onto $\frac{I(M)}{\mathcal D(M)}$, this is also finite, say of cardinality $N>1$ since $\frac{I(M)}{\mathcal D(M)}\neq 0$. Write
$$
K:=\left\{x\in \pi_1(B\bdiffh(M))\cong\pi_1(B\bdiffh(L)): \partial x=0\in \frac{I(M)}{\mathcal D(M)}\right\}.
$$
We now have two extensions of finite groups
$$
\begin{tikzcd}[remember picture]
    0\rar & H_1(C_2; \Wh(L))/\im\partial\dar["\cong"]\rar & \pi_1(B\bdiff(L))\rar[two heads] & \pi_1(B\bdiffh(L))\rar &1\\
    0\rar & H_1(C_2; \Wh(M))/\im\partial\rar & \pi_1(B\bdiff(M))\rar[two heads] & K\rar &1,
\end{tikzcd}
$$
where the left vertical isomorphism is a consequence of Proposition \ref{leftextensionlemma}. Even though $\partial$ is a crossed homomorphism, we still have that $N\cdot |K|=|\pi_1(B\bdiffh(M))| =|\pi_1(B\bdiffh(L))|$, and therefore the result follows. \begin{tikzpicture}[overlay,remember picture]
\path (\tikzcdmatrixname-1-4) to node[midway,sloped]{$\supset$}
(\tikzcdmatrixname-2-4);
\end{tikzpicture}
\end{proof}

\subsection{\texorpdfstring{The candidate $W:L\hcob M$}{The candidate W:L->M}}\label{candidatesection}
Let $L^{2n-1}_{p}(r_1:\dots:r_n)$ denote the linear lens space with fundamental group $C_p$ and weights $r_1,\dots,r_n\mod p$, i.e., the quotient of the sphere $S^{2n-1}$ by the free (left) $C_p$-action given by
$$
t\cdot (z_1,\dots,z_n):=(\zeta^{r_1}z_1,\dots, \zeta^{r_n}z_n), \qquad (z_1,\dots,z_n)\in S^{2n-1}\subset \C^n,
$$
where $t\in C_p$ is the generator and $\zeta=\exp(2\pi\mathsf{i}/p)$. We identify $\pi_1(L^{2n-1}_p(r_1:\dots:r_n))$ with $C_p$ by sending the homotopy class represented by the loop 
$$
[0,1]\longrightarrow L^{2n-1}_p(r_1:r_2:\dots: r_n), \quad s\longmapsto [\zeta^{s\cdot r_1},0,\dots,0]
$$
to $t\in C_p$. The goal of this section is to prove

\begin{thm}\label{ThmAiequivalent}
Let $L$ be the lens space $L^{12k-1}_{7}(r_1:\dots: r_{6k})$ with
$$
r_1=\dots=r_k=1, \qquad r_{k+1}=\dots=r_{2k}=2, \quad \dots \quad r_{5k+1}=\dots=r_{6k}=6 \mod 7.
$$
The element $u:=2+2t-t^3-t^4-t^5$ is a unit in $\Z[C_7]$ with inverse $u^{-1}=1-2t+3t^2-3t^3+3t^4-2t^5+t^6$, and hence represents an element of $\Wh(L)$. Then the $h$-cobordism $W: L\hcob M$ with torsion $\tau(W,L)=u$ satisfies conditions \ref{condI} and \ref{condII} of Proposition \ref{conditionprop} with $N=\left|\frac{I(M)}{\mathcal D(M)}\right|=3$. In particular
$$
|\pi_1(B\bdiff(L))|=3\cdot|\pi_1(B\bdiff(M))|<\infty,
$$
and Theorem \ref{diffdifference}$(i)$ holds.
\end{thm}

The proof of Theorem \ref{ThmAiequivalent} will be established in Propositions \ref{propIthmAi} and \ref{finitemcg} below. 
\begin{prop}\label{propIthmAi}
The $h$-cobordism $W:L\hcob M$ of Theorem \ref{ThmAiequivalent} satisfies condition \ref{condI} of Proposition \ref{conditionprop}. In fact,
$$
\left|\frac{I(M)}{\mathcal D(M)}\right|=3.
$$
\end{prop}
\begin{proof} The algebraic involution $\bar{\cdot}: \Wh(\pi)\to \Wh(\pi)$ is trivial when $\pi$ is a finite abelian group \cite[Prop. 4.2]{BassInvolution}. Therefore by (\ref{doublesubgroup}), the subgroups of double $h$-cobordisms $\mathcal D(L)$ and $\mathcal D(M)$ are trivial since $L$ and $M$ are odd-dimensional and orientable (so their first Stiefel--Whitney classes vanish), and $\pi_1(L)\cong\pi_1(M)\cong C_7$ is certainly finite abelian. It thus suffices to show that $I(L)=0$ and $|I(M)|=3$. 

The first assertion follows from \cite[Cor. 12.12]{MilnorWhiteheadTorsion}. We now prove that $I(M)\neq 0$, i.e., we construct non-trivial inertial $h$-cobordisms starting at $M$. For a diffeomorphism $\phi\in \diff(L)$, write $V_\phi$ for the inertial $h$-cobordism
$$
V_\phi:=W\circ M_{\phi^{-1}}\circ \overline{W}: M\hcob L\scob L\hcob M,
$$
i.e., $V_\phi=\overline{W}\cup_{\phi^{-1}} W\in I(M)$. Observe that $h^{\overline W}: M\overset{\simeq}\longrightarrow L$ is homotopy inverse to $h^W: L\overset{\simeq}\longrightarrow M$ because $\overline W\circ W=D(W)$ and $W\circ\overline W=D(\overline W)$ are $h$-cobordant rel boundary to the trivial $h$-cobordisms $L\times I$ and $M\times I$, respectively (see Figure \ref{doublehcobs}). So $h^{\overline W\circ W}=h^{\overline W}\circ h^W$ is homotopic to $h^{L\times I}=\mathrm{Id}_L$ (and similarly $h^W\circ h^{\overline W}\simeq \mathrm{Id}_M$). The $h$-cobordism $V_\phi$ then has torsion
\begin{align*}
\tau(V_\phi,M)&=\tau(\overline{W},M)+(h^{\overline W})_*^{-1}\tau(W\circ M_{\phi^{-1}}, L)\\
&=(-1)^{12k-1}h^W_*\overline u+h^W_*\phi_*u\\
&=h^{W}_*(\phi_*u-u),
\end{align*}
where we have used that $\overline u=u$ by the triviality of the algebraic involution. Therefore, if we are able to find self-diffeomorphisms $\phi$ of $L$ for which $\phi_*u\neq u$ in $\Wh(L)$, then we will have achieved our task.

\begin{figure}[h]
    \centering
    \includegraphics[scale=0.13]{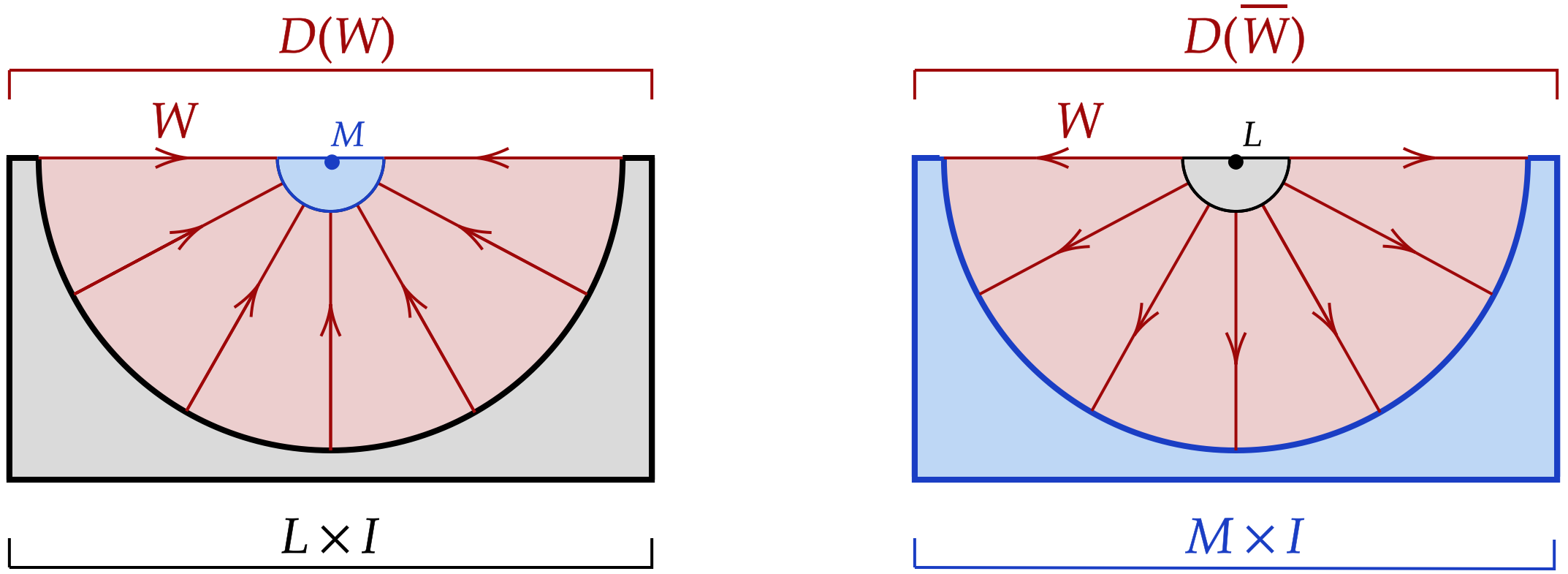}
    \caption{$h$-cobordisms rel boundary $D(W)\hcob L\times I$ and $D(\overline W)\hcob M\times I$.}
    \label{doublehcobs}
\end{figure}

\begin{cl}
There are orientation-preserving self-diffeomorphisms $\phi_i: L\overset{\cong}\longrightarrow L$ for $i\in (\Z/7)^\times$ such that $\pi_1(\phi_i): t\mapsto t^{i}$.
\end{cl}
\begin{proof}[Proof of Claim]
In fact by the main theorem of \cite{HsiangJahren} (see (\ref{pi0Lses}) below), the natural map $\pi_0(\diff(L))\to \pi_0(s\mathrm{Aut}(L))$ is surjective, so it will suffice to find simple homotopy automorphisms $f_i\in s\mathrm{Aut}(L)$ for each $i\in (\Z/7)^\times$ such that $\pi_1(\hspace{2pt} f_i): t\mapsto t^{i}$. By \cite[Thm. II.b \& Thm. V]{Olum1953}, the natural map $\gamma:\pi_0(h\mathrm{Aut}(L))\to \mathrm{Aut}(\pi_1 L)\cong (\Z/7)^\times$ is injective with  image
    $$
    \im \gamma=\{i\in \Z/7: i^{\hspace{1pt} 6k}\equiv \pm 1 \mod 7\}.
    $$
The classes $[\hspace{1pt}f]$ sent to $i\in (\Z/7)^\times$ with $i^n\equiv +1\mod 7$ (resp. $i^n\equiv -1\mod 7$) are orientation-preserving (resp. orientation-reversing). But if $i\in (\Z/7)^\times$, $i^6\equiv 1\mod 7$ and so $i^{6k}\equiv 1\mod 7$ too. Therefore for each $i\in (\Z/7)^\times$, there exists some (orientation-preserving) homotopy automorphism $f_i\in h\mathrm{Aut}(L)$ such that $\pi_1(\hspace{2pt} f_i):t\mapsto t^{i}$. By \cite[Lem. 12.5]{MilnorWhiteheadTorsion}, $f_i$ is a simple homotopy automorphism if and only if
$$
(\hspace{2pt} f_i)_*\Delta(L)=\Delta(L)\in \Q[C_7]/\sim,
$$
where $\Delta(L)$ denotes the \textit{R-torsion} of $L$ \cite[Lem. 12.4]{MilnorWhiteheadTorsion}. Here, two elements $x,y\in \Q[C_7]$ are related $x\sim y$ if and only if there exists some $g\in C_7$ such that $x=\pm g\cdot y$. Recall also \cite[p. 406]{MilnorWhiteheadTorsion} that the $R$-torsion of $L$ is
$$
\Delta(L)=\prod_{j=1}^{6k}(t^{r_j}-1)=\prod_{j=1}^6(t^{\hspace{1pt}j}-1)^k,
$$
and so
$$
(\hspace{2pt} f_i)_*\Delta(L)=(\hspace{2pt} f_i)_*\left(\prod_{j=1}^6(t^{\hspace{1pt}j}-1)^k\right)=\prod_{i=1}^6(t^{i\cdot j}-1)^k=\Delta(L), \quad i\in (\Z/7)^\times.
$$
Therefore $f_i$ is a simple homotopy equivalence for $i\in (\Z/7)^\times$, as claimed. 
\end{proof}

Now it is easily checked that $(\phi_6)_*u=u$, so $(\phi_2)_*u=(\phi_5)_*u$ and $(\phi_3)_*u=(\phi_4)_
*u$ in $\Q[C_7]/\sim$. On the other hand, the three non-trivial units
$$
u=2+2t-t^3-t^4-t^5, \quad (\phi_2)_*u=2+2t^2-t^6-t-t^3, \quad (\phi_3)_*u=2+2t^3-t^2-t^5-t
$$
represent different elements in $\mathrm{Wh}(L)$ (for instance, the difference between the powers of $t$ of the terms whose coefficient is $2$ in the three units is different mod $7$). By our previous argument, the $h$-cobordisms $V_{\phi_1}=M\times I$, $V_{\phi_2}$ and $V_{\phi_3}$ are pairwise non-diffeomorphic inertial $h$-cobordisms starting at $M$, so $|I(M)|\geq 3$. Note that $V_{\phi_i}\cong V_{\phi_{7-i}}$ for $i=1,\dots 6$. Conversely, suppose that $V: M\hcob M$ is an inertial $h$-cobordism (by possibly post-composing with a mapping cylinder, we may assume that the target of $W$ is $M$ itself). Then the inertial $h$-cobordism $\overline W\circ V\circ W:L\hcob L$ must be trivial as $I(L)=0$, and $(h^{\overline W\circ V\circ W})^{-1}_*=(h^W)^{-1}_*(h^{V})_*^{-1}h^W_*=(\phi_i)_*$ for some $i\in (\Z/7)^\times$ because $\mathrm{Aut}(\pi_1L)\cong (\Z/7)^\times$. Using $\tau(\overline W\circ V\circ W,L)=0$ we get that
$$
\tau(V,M)=(h^V)^{-1}_*h^W_*u-h^W_*u=h_*^W((\phi_i)_*u-u)=\tau(V_{\phi_i},M),
$$
i.e. $V$ is diffeomorphic to $V_{\phi_i}$ rel $M$. Hence, $|I(M)|=3$ and this finishes the proof.
\end{proof}

\begin{rem}\label{remarkinertianotsubgroup}
This is an example of how badly-behaved the set of inertial $h$-cobordisms of a manifold may be. For instance in the case at hand, it is not an $h$-cobordism invariant. It is also \textit{not} a subgroup of $\Wh(M)\cong\Wh(C_7)\cong \Z^2$ (see \cite[$\S7$]{BassWhiteheadGroups} and \cite[pp. 202--205]{SteinWhiteheadgroups}) because $I(M)$ is a finite subset with cardinality strictly greater than $1$. See \cite[Rmk. 6.2]{HausmannInertialhcobs} for more instances of this phenomenon.
\end{rem}

\begin{warn}
The main theorem of \cite{Kwasik1} states that $I(M)=0$ if $M$ is a fake lens space, that is, the orbit space of a free (possibly non-linear) action of a finite cyclic group on a sphere. This clearly contradicts Proposition \ref{propIthmAi}, but we believe that the proof of the result in \cite{Kwasik1} is fallacious: following a trail of references \cite[Claim 2]{Kwasik1}, \cite[Prop. 3.2]{Kwasik2} and \cite[p. 353]{Kwasik3}, it is eventually stated that if $\tau\in \Wh(M)$ is the torsion of some inertial $h$-cobordism of $M$, then $2\tau=0$. This supposedly follows from the proof of \cite[Prop. 12.8]{MilnorWhiteheadTorsion}, but in the statement of this result it is required that the $h$-cobordism between special manifolds be compatible with the given identifications of the fundamental groups, i.e., that $h^W_*: \Wh(M)\to \Wh(M')$ has been trivialised beforehand. This requirement does not hold in the case of \cite[p. 353]{Kwasik3}, and is exactly what we exploit in the proof of Proposition \ref{propIthmAi}.
\end{warn}

We now deal with \ref{condII}. Since the natural map $\pi_0(\diff(L))\to \pi_0(\bdiff(L))\cong \pi_1(B\bdiff(L))$ is surjective, it suffices to show

\begin{prop}\label{finitemcg}
The mapping class group $\Gamma(L)=\pi_0(\diff(L))$ of $L$ is finite. In particular, $W:L\hcob M$ satisfies condition \ref{condII} of Proposition \ref{conditionprop}.
\end{prop}
\begin{proof}
According to the main result of \cite{HsiangJahren}, the mapping class group of $L$ fits into an extension of groups
\begin{equation}\label{pi0Lses}
\begin{tikzcd}
0\rar &Q\oplus H\rar &\pi_0(\diff(L))\rar & \pi_0(s\mathrm{Aut}(L))\rar & 0.
\end{tikzcd}
\end{equation}
The group $H$ is the image of $[\Sigma L_+, \mathrm{Top}/O]_*$ in $[\Sigma L_+, G/O]_*$. By \cite[Thm. 1.1]{HsiangSharpe}, the group $Q$ also appears in an exact sequence
    \begin{equation}\label{hsiangsharpe}
    L_{2n+2}^s(\Z[C_7])\longrightarrow H_0(C_2;\pi_0(C(L)))\longrightarrow Q\longrightarrow L^s_{2n+1}(\Z[C_7]),
    \end{equation}
    where $\pi_0(C(L))$ is the group of (isotopy classes) of pseudoisotopies of $L$. By a computation of Hatcher--Wagoner \cite{HatcherWag} and corrections of Igusa \cite{IgusaCorrection}, it fits in yet another exact sequence
    \begin{equation}\label{HatcherWagonerExactSeq}
        \begin{tikzcd}
            \Wh_1^+(C_7;\mathbb{F}_2)\rar& \pi_0(C(L))\rar & \Wh_2(C_7) \rar &0
        \end{tikzcd}
    \end{equation}
    as long as $\dim L\geq 7$; this is our case. In (\ref{hsiangsharpe}) and (\ref{HatcherWagonerExactSeq}), for an abelian group with involution $\pi$,
    \begin{itemize}
    \item $L^s_*(\Z\pi)$ are the simple quadratic $L$-groups of the group ring $\Z\pi$, 
    \item $\Wh_2(\pi)$ is the abelian group (with involution) defined as the cokernel of the map arising from algebraic $K$-theory (see (\ref{WaldhausenWhFibseq}))
  \begin{equation}\label{Wh2defn}
    \pi_2^s((B\pi)_+)\longrightarrow K_2(\Z\pi)\longrightarrow \Wh_2(\pi)\longrightarrow 0,
    \end{equation}
    
    \item $\Wh_1^+(\pi;\mathbb{F}_2):=H_0(C_2; \mathbb{F}_2[\pi])$ with a certain involution.
\end{itemize}

We now show that each of the groups in the extension (\ref{pi0Lses}) are finite. As mentioned in the proof of Proposition \ref{propIthmAi}, $\pi_0(s\mathrm{Aut}(L))\subset \pi_0(h\mathrm{Aut}(L))\subset (\Z/7)^\times$, so it is definitely finite. 

By Proposition \ref{Ktheoryappendixprop}, $K_2(\Z[C_7])$ is finite, and hence so is $\Wh_2(C_7)$ by (\ref{Wh2defn}). Since $\mathbb{F}_2[C_7]$ is finite, $\Wh_1^+(C_7;\mathbb{F}_2)$ is so too. Moreover, the (simple) $L$-theory of $\Z\pi$ for finite groups $\pi$ of odd order is zero in odd degrees \cite[Thm. 1]{BaksimpleLtheory}. Therefore $L^s_{2n+1}(\Z[C_7])=0$, and thus $Q$ is finite by (\ref{hsiangsharpe}).

The finiteness of $H$ is a consequence of the following two observations: firstly that the infinite loop space $\mathrm{Top}/O$ has finite homotopy groups in every degree (see \cite[Thm. 5.5]{KirbySieb}). Secondly that if $X$ is a (pointed) finite CW-complex and $Y$ a (pointed) space with finite homotopy groups in every degree, then the set $[X,Y]_*$ of (pointed) maps from $X$ to $Y$ up to homotopy is finite---indeed, this follows easily by induction on the skeleta $\{X_k\}_{k\geq 0}$ of $X$ by considering the cofibre sequences
$$
\begin{tikzcd}
X_{k-1}\rar[hook] &X_k\ar[r, two heads] &\bigvee_{i\in I_k} S^k,
\end{tikzcd}
$$
where $I_k$ is a finite set (because $X$ is a finite CW-complex). Hence $[\Sigma L_+, \mathrm{Top}/O]_*$ is finite, and thus so is $H$. This finishes the proof.
\end{proof}

\begin{rem}\label{smoothcategoryremark}
For $CAT=\mathrm{Top}$ or $PL$, the group $H$ of (\ref{pi0Lses}) should be replaced by the image of $[\Sigma L_+, \mathrm{Top}/CAT]_*$ in $[\Sigma L_+, G/CAT]$, which readily vanishes for $CAT=\mathrm{Top}$ and is seen to be finite too for $CAT=PL$ (see e.g. \cite{Brumfiel1968}), so the same argument in the proof of Proposition \ref{finitemcg} goes through. 
\end{rem}

\section{\texorpdfstring{Proof of Theorem \ref{diffdifference}($ii$)}{Proof of Theorem A(ii)}}\label{section5}
In this section we finish the proof of Theorem \ref{diffdifference} using the candidate $W:L\hcob M$ of 
Theorem \ref{ThmAiequivalent}. We would hope that $\diff(L)$ and $\bdiff(L)$ differ as much as $\diff(M)$ and $\bdiff(M)$ do, so that the difference of block mapping class groups established in Theorem \ref{ThmAiequivalent} carries over to the $\diff$-level. This is the case in some range and \textit{up to extensions} \cite[Thm. A]{WWI}.

\begin{thm}[Weiss--Williams]\label{WWbdiffmoddiff} Let $M^d$ be compact a smooth $d$-manifold. There exists a map
$$
\Phi^s: \bdiff/\diff(M)\longrightarrow \Omega^{\infty}\left(\Hsp_\diff^s(M)_{hC_2}\right)
$$
which is $(\phi_M+1)$-connected, where $\phi_M$ denotes the concordance stable range of $M$ (which by Igusa's theorem \cite{IgusaConcordanceStability} is at least $\min(\frac{d-4}{3},\frac{d-7}{2})$). 
\end{thm}
\begin{rem}\label{remafterWWIThm}
The $C_2$-spectrum $\Hsp^s_{\diff}(M)$, known as the (smooth) \textit{$s$-cobordism spectrum of $M$}, is the $1$-connective cover of the (non-connective smooth) \textit{$h$-cobordism spectrum} $\Hsp_{\diff}(M)$. This latter spectrum is roughly built out of deloopings of spaces of $h$-cobordisms (cf. \cite[Lem. 1.12]{WWI}) and its infinite loop space $\mathcal{H}_\diff(M)$, the \textit{space of stable $h$-cobordisms}, coincides with that of $\Sigma^{-1}\Whsp(M)$ by the stable parametrised $h$-cobordism theorem of Waldhausen--Jahren--Rognes \cite{WaldhausenParametrisedhcob}. Here $\Whsp(M)$ stands for the \textit{smooth Whitehead spectrum} of $M$ (see Section \ref{whspappendixsection}). Moreover, the negative homotopy groups of these two spectra abstractly coincide (cf. \cite[Cor. 5.6]{WWI}), which lead Weiss and Williams to rename $\Hsp_\diff(M)$ by $\Sigma^{-1}\Whsp(M)$ (though conjecturally true, this was not fully justified).

The only property we will use about $\Hsp(-)$ is that its homotopy groups (ignoring the involution) are invariants of the homotopy type of $(-)$, as those of $\Sigma^{-1}\Whsp(-)$ are. In particular, if $W: L\hcob M$ is an $h$-cobordism, there is an isomorphism of groups
\begin{equation}\label{hcobinvarianceinvolution}
    \pi_*^s(\Hsp_{\diff}(L))\cong \pi_*^s(\Hsp_{\diff}(M)).
\end{equation}
We will not need to analyse the involutions in $\Hsp(L)$ and $\Hsp(M)$, but one can show that these two $C_2$-spectra are equivalent (in fact, the homotopy type of $\Hsp(-)$ is nearly an invariant of the tangential homotopy type of $(-)$; see \cite[\S 5]{MElongknots} for more details on the involutions in $\Hsp_\diff(M)$ and $\Whsp(M)$). It is also not difficult to see that the involution on $\pi_0(\Hsp_\diff(M))\cong \Wh(M)$ corresponds to the rule $\tau\mapsto (-1)^{d-1}\overline{\tau}$ (cf. \cite[Cor. 5.8]{MElongknots}), which fits well with Theorem \ref{ThmB}. We expand on the relation between Theorems \ref{ThmB} and \ref{WWbdiffmoddiff} in Section \ref{whspappendixsection}. 
\end{rem}

Now since $d=12k-1\geq 11$ (so $\phi_M+1\geq 2$), it follows from Theorem \ref{WWbdiffmoddiff} that $\pi_1(\bdiff/\diff(L))\cong\pi_1^s(\Hsp^s_\diff(L)_{hC_2})$. As $\Hsp_\diff^s(L)$ is $1$-connective, its homotopy fixed point spectral sequence (cf. \cite{Bousfield1972}) then yields isomorphisms
\begin{equation}\label{pi1pseudoisotopy}
\begin{tikzcd}[row sep = 4pt]
\pi_1(\bdiff/\diff(L))\cong H_0(C_2; \pi_1^s(\Hsp_\diff(L))),\\
\pi_1(\bdiff/\diff(M))\cong H_0(C_2; \pi_1^s(\Hsp_\diff(M))),
\end{tikzcd}
\end{equation}
for potentially different $C_2$-actions on $\pi_1^s(\Hsp_\diff(L))\cong\pi_1^s(\Hsp_\diff(M))$. Consider the extensions
\begin{equation}\label{bdiffmoddiffdiagramLM}
\begin{tikzcd}[row sep = 5pt]
\pi_1(\bdiff/\diff(L))\rar["\partial"] & \pi_0(\diff(L))\rar[two heads] &\pi_0(\bdiff(L))\rar & 0,\\
\pi_1(\bdiff/\diff(M))\rar["\partial"] & \pi_0(\diff(M))\rar[two heads] &\pi_0(\bdiff(M))\rar & 0.
\end{tikzcd}
\end{equation}
We know from Theorem \ref{ThmAiequivalent} that $|\pi_0(\bdiff(L)|=3\cdot |\pi_0(\bdiff(M))|$, so in order to prove Theorem \ref{diffdifference}($ii$) it suffices to establish the next result.
\begin{prop}\label{finalprop}
The groups $\pi_1(\bdiff/\diff(L))$ and $\pi_1(\bdiff/\diff(M))$ are finite and their cardinality is not divisible by $3$. Together with Theorem \ref{ThmAiequivalent}, it follows that the $3$-adic valuations of $|\Gamma(L)|$ and $|\Gamma(M)|$ differ. This proves Theorem \ref{diffdifference}$(ii)$.
\end{prop}

\begin{proof}
Given (\ref{hcobinvarianceinvolution}) and (\ref{pi1pseudoisotopy}), we need only verify the first claim for the abelian group $\pi_1^s(\Hsp_\diff(L))\cong\pi_1(\mathcal H_\diff(L))$. Since $\dim L=d\geq 11$, it follows from Igusa's lower bound on the concordance stable range that $\pi_1(\mathcal H_\diff(L))=\pi_0(C(L))$, which fits in the extension (\ref{HatcherWagonerExactSeq}).

We have already argued in the proof of Proposition \ref{finitemcg} that both of the groups $\Wh_2(C_7)$ and $\Wh_1^+(C_7;\mathbb{F}_2)$ in the extension are finite. Both summands are moreover $3$-locally trivial ($\Wh_1^+(C_7;\mathbb{F}_2)$ is $2$-tosion, and $\Wh_2(C_7)$ is a quotient of $K_2(\Z[C_7])$, which has no $3$-torsion by Proposition \ref{Ktheoryappendixprop}). Any quotient of this group (e.g. those of (\ref{pi1pseudoisotopy})) will have this same property, so the result follows. The proof of Theorem \ref{diffdifference} is now complete.
\end{proof}

\begin{rem}\label{diffremarkno2}
For $CAT=\mathrm{Top}$ or $PL$, the $h$-cobordism spectrum $\Hsp_\diff(L)$ should be replaced by its topological version $\Hsp_{\mathrm{Top}}(L)$ (this is in fact the one that appears originally in \cite{WWI}). To argue that $\pi_1^s(\Hsp_{\mathrm{Top}}(L))\cong\pi_2^s(\Whsptop(L))$ is finite and $3$-local as in the previous proof, we consider the diagram of cofibre sequences of spectra (see (\ref{WaldhausenWhFibseq}) and (\ref{WaldhausenWhFibseqtop}))
$$
\begin{tikzcd}
\mathbb{S}\wedge L_+\dar\rar["\iota"] &\mathbf{A}(L)\dar[equal]\rar[two heads] & \Whsp(L)\dar\\
\mathbf{A}(*)\wedge L_+\dar[two heads]\rar["\alpha"] &\mathbf{A}(L)\dar\rar[two heads] & \Whsptop(L)\dar[two heads]\\
\Whsp(*)\wedge L_+\rar &*\rar &\Sigma\Whsp(*)\wedge L_+.
\end{tikzcd}
$$
Then $\pi_2^s(\Whsptop(L))$ will be finite and $3$-locally trivial if $\pi_2^s(\Sigma\Whsp(*)\wedge L_+)\cong \pi_1^s(\Whsp(*)\wedge L_+)$ is. This in turn follows from the Atiyah--Hirzebruch spectral sequence, as $\Whsp(*)\simeq \Whsp(D^5)$ by the homotopy invariance of the Whitehead spectrum, and because the latter is $1$-connective by the $s$-cobordism theorem (in fact it is $2$-connective by Cerf's pseudoisotopy theorem).
\end{rem}

\begin{rem}[Another possible example] Theorem \ref{diffdifference} may also holds for the lens space $L=L^{8k-1}_5(r_1:\dots:r_{4k})$, where 
$$
r_1=\dots=r_k=1,\qquad r_{k+1}=\dots=r_{2k}=2, \quad \dots \quad r_{3k+1}=\dots=r_{4k}=4\mod 5,
$$
and the $h$-cobordism $W: L\hcob M$ with $\tau(W,L)=[1-t-t^4]\in \Wh(C_5)$. The argument for part $(i)$ is exactly analogous to that of Subsection \ref{candidatesection}, but part $(ii)$ is trickier. The inertia set $I(M)$ will have size two (instead of three), and the group $\Wh_1^+(C_5;\mathbb{F}_2)$ does have $2$-torsion. The alternative then is to show directly that the map $\partial$ in (\ref{bdiffmoddiffdiagramLM}) is injective by identifying $\pi_1(\bdiff/\diff(L))$ with the cobordism group $\pi_0(\mathcal B(L))$ of \cite[p.1]{HsiangJahren}. However, this argument does rely on the claim made in the proof of \cite[Sublemma 4.2]{HsiangJahren} that a certain map $H_0(C_2; \Wh_2(C_5))\to L^{\mathrm{St}}_{8k-1}(C_5)$ is injective when inverting the prime $2$. We do not know how to prove this, nor have we found a reference that does.
\end{rem}

\appendix
\section{\texorpdfstring{An algebraic $K$-theory computation}{An algebraic K-theory computation}}\label{AppendixA}

The aim of this section is to prove the following.

\begin{prop}\label{Ktheoryappendixprop}
For $p$ a prime, $K_2(\Z[C_p])$ is finite. Moreover when $p=7$, its $3$-torsion part vanishes:
$$
K_2(\Z[C_7])_{(3)}=0.
$$
\end{prop}

\begin{rem}
    The author would like to thank John Nicholson for making him aware of the paper \cite[Thm. 2.7]{Zhang2019} which, taken together with the computation in \cite[Thm. 1.1]{K2algnumberfields}, easily implies Proposition \ref{Ktheoryappendixprop} when $p=7$; this is the only case needed in the proof of Theorem \ref{diffdifference}. By the time we became aware of this fact, we had already come up with an alternative proof, which we believe to be a nice application of a celebrated result of Land--Tamme. For this reason, we still present our original proof below, but the pragmatic reader may wish to skip this section.
\end{rem}

The main ingredient of this computation is the main theorem of Land--Tamme \cite{MarkusLandKtheory}: Given a Milnor square of ring (spectra)
$$
\begin{tikzcd}
A\rar\dar & B\dar\\
A'\rar & B',
\end{tikzcd}
$$
i.e., a pullback square of ring spectra with $\pi_0(B)\to \pi_0(B')$ surjective, they functorially associate a connective ring spectrum $\mathcal R$ for which there is a Mayer--Vietoris sequence for algebraic $K$-theory
\begin{equation}\label{KMayVie}
\begin{tikzcd}
    \dots\rar & K_{i+1}(\mathcal R)\rar & K_i(A)\rar &K_i(A')\oplus K_i(B)\rar &K_i(\mathcal R)\rar & \dots
\end{tikzcd}
\end{equation}
for every $i\in \Z$. Moreover, there is an equivalence of spectra $\mathcal R\to A'\otimes_A B$ (but \textit{not} of $\mathbb E_1$-rings in general) and a map of $\mathbb{E}_1$-rings $\mathcal R\to B'$. For $p$ a prime, the pullback square we will consider is
\begin{equation}\label{pbsquarerings}
    \begin{tikzcd}
    \Z[C_p]\cong \Z[t]/(1-t^p)\rar\dar["t=1"] & \Z(\zeta_p)\cong \Z[t]/(1+t+\dots+t^{p-1})\dar["t=1"]\\
    \Z\rar["\mod p"] & \Z/p,
    \end{tikzcd}
\end{equation}
or rather that induced by applying the Eilenberg--MacLane functor $H(-)$ to (\ref{pbsquarerings}). A straight-forward computation of $\mathrm{Tor}^{\Z[C_p]}_i(\Z,\Z(\zeta_p))$ shows that
$$
\pi_i^s(\mathcal R)\cong\left\{\begin{array}{cc}
    \Z/p,  & i=2k\geq 0,   \\
    0, & \text{otherwise}.
\end{array}
\right.
$$
Hence, the natural map $\mathcal R\to H\Z/p$ is an isomorphism on $\pi_0$ and a $\Z[1/p]$-equivalence of connective $\mathbb E_1$-rings. Therefore by \cite[Lem. 2.4]{MarkusLandKtheory}, it induces an isomorphism of localised $K$-theory $K_*(\mathcal R)\otimes \Z[1/p]\cong K_*(\Z/p)\otimes \Z[1/p]$. A portion of the exact sequence (\ref{KMayVie}) localised away from $p$ thus reads
\begin{equation}\label{MVZp}
\begin{tikzcd}[column sep = 12pt]
\Big(K_3(\Z)\oplus K_3(\Z(\zeta_p))\rar & K_3(\Z/p)\rar & K_2(\Z[C_p])\rar & K_2(\Z)\oplus K_2(\Z(\zeta_p))\Big)\otimes \Z\Big[\frac{1}{p}\Big].
\end{tikzcd}
\end{equation}
We first analyse the (3-adic part of the) map $K_3(\Z)\to K_3(\Z/p)$ for $p\neq 3$.

\begin{lem}\label{K3Fplemma}
The map $K_3(\Z)_{(3)}\to K_3(\Z/p)_{(3)}$ is injective for $p\neq 3$.
\end{lem}
\begin{proof}
According to \cite[Claim 4]{QuillenLetter}, for every integer $k\geq 1$ and odd prime $\ell\neq p$, the composition $\pi_{4k-1}^s\to K_{4k-1}(\Z)\to K_{4k-1}(\Z/p)$ is injective on $\mathrm{Im}(J:\pi_{4k-1}(O)\to \pi_{4k-1}^s)_{(\ell)}$: indeed, Diagram 4 loc. cit. is the commutative diagram
$$
\begin{tikzcd}
\pi_{4k-1}^s\rar\dar["-e"'] & K_{4k-1}(\Z/p)\cong \Z/(p^{2k}-1)\dar[hook, "\theta"']\\
\Q/a_k\Z\rar & \Q/\Z[\tfrac{1}{p}],
\end{tikzcd}
$$
where $e$ denotes Adams' invariant, which is injective on $\operatorname{Im} J$, $\theta$ is injective with image the unique subgroup of order $p^{2k}-1$, and $a_k$ is $1$ or $2$ depending on whether $k$ is odd or even, respectively. The lower horizontal map is the natural one, which is injective on $\ell$-torsion if $2p$ does not divide $\ell$.

For $k=1$, the image of the $J$-homomorphism is the whole of $\pi_3^s\cong \Z/24$, $K_3(\Z/p)\cong \Z/(\hspace{1pt} p^2-1)$ by \cite[Thm. 8($i$)]{QuillenKtheoryfinfields}, and $K_3(\Z)\cong \Z/48$ by \cite{K3Z}. Noting that $3\mid p^2-1$ if $p\neq 3$ is prime, the result readily follows from the previous claim  when $\ell=3$. 
\end{proof}

\begin{proof}[Proof of Proposition \ref{Ktheoryappendixprop}]
It is well known that $K_2(\Z)\cong \Z/2$ \cite[Cor. 10.2]{MilnorKtheory}, $K_3(\Z/p)\cong \Z/(\hspace{1pt} p^2-1)$ and $K_2$ of the ring of integers of a number field is finite \cite[Thm. 1]{Quillenalgebraicnumberfields}, \cite[Prop. 12.2]{BorelKtheorynumberfield} (in particular $K_2(\Z(\zeta_p))$ is). A very similar argument to \cite[Lem. 2.4]{MarkusLandKtheory} replacing the Serre class of $\Lambda$-local abelian groups with the Serre class of finitely generated abelian groups shows that the map $K_3(\mathcal R)\to K_3(\Z/p)$ is an equivalence mod this Serre class, so as $K_3(\Z/p)\cong \Z/(\hspace{1pt} p^2-1)$ is finitely generated, so is $K_3(\mathcal R)$. In fact since $K_3(\mathcal R)$ is finitely generated and $K_3(\mathcal R)\otimes \Z[1/p]\cong K_3(\Z/p)\otimes \Z[1/p]\cong \Z/(\hspace{1pt} p^2-1)$ is finite, $K_3(\mathcal R)$ is finite too. It follows from (\ref{MVZp}) that $K_2(\Z[C_p])$ is finite for every $p$.

Let now $p=7$ so that $K_3(\Z/7)\cong \Z/48$, and hence by Lemma \ref{K3Fplemma}, the map $\Z/3\cong K_3(\Z)_{(3)}\to K_3(\Z/7)_{(3)}\cong \Z/3$ is an isomorphism. Now $K_2(\Z(\zeta_7))=\Z/2$ \cite[Thm. 1.1]{K2algnumberfields}, and localising (\ref{MVZp}) at the prime $3$($\neq p=7$) we get that $K_2(\Z[C_7])_{(3)}=0$.
\end{proof}

\section{Connections to Weiss--Williams I}\label{appendixB}

\subsection{\texorpdfstring{The group of $h$-block diffeomorphisms $\bdiffh(M)$}{The group of h-block diffeomorphisms BDiffh(M)}}
Recall that $\bdiffb(M\times \R)_\bullet$ denotes the semi-simplicial group of block diffeomorphisms of $M\times \R$ \textit{bounded in the $\R$-direction}---a $p$-simplex consists of a face-preserving (cf. Definition \ref{stratifieddefn}) diffeomorphism $\phi: M\times \R\times \Delta^p\overset{\cong}\longrightarrow_\Delta M\times \R\times \Delta^p$ such that there exists some positive constant $K>0$ with $|\mathrm{pr}_\R\phi(x,t,v)-t|<K$ for all $(x,t,v)\in M\times \R\times \Delta^p$. In this section we prove
\begin{prop}\label{appendixBprop}
For $d=\dim M\geq 5$, there is a zig-zag of weak equivalences of Kan semi-simplicial sets
$$
\begin{tikzcd}
\Omega B\bdiffh(M)_\bullet &\bdiffb_{>1/2}(M\times\R)_\bullet\lar["\simeq","\mathcal R_\bullet"']\rar[hook, "\simeq"] &\bdiffb(M\times \R)_\bullet.
\end{tikzcd}
$$
In particular, there are homotopy equivalences
$$
\bdiffh(M):=|GB\bdiffh(M)_\bullet|\simeq|\Omega B\bdiffh(M)_\bullet|\simeq \bdiffb(M\times\R).
$$
\end{prop}
Let us explain the new notation. Recall that the simplicial loop space $\Omega B\bdiffh(M)_\bullet$ has as $p$-simplices those $(\hspace{1pt} p+1)$-simplices $W\Rightarrow \Delta^{p+1}$ of $B\bdiffh(M)_\bullet$ with $W_0=M$ and $\partial_0W=M\times \Delta^p$. The sub-semi-simplicial set $\bdiffb_{>1/2}(M\times \R)_\bullet\subset \bdiffb(M\times \R)_\bullet$ has as $p$-simplices those bounded diffeomorphisms $\phi: M\times \R\times \Delta^p\overset{\cong}{\longrightarrow}_\Delta M\times\R\times \Delta^p$ with
\[
\phi(M\times(1/2,\infty)\times \Delta^{p})\subset M\times(1/2,\infty)\times\Delta^p.
\]
The map $\mathcal R_\bullet$ sends a diffeomorphism $\phi\in \bdiffb_{>1/2}(M\times\R)_p$ to the region in $M\times \R\times \Delta^p$ enclosed by $M\times\{0\}\times\Delta^p$ and $\phi(M\times \{1\}\times \Delta^p)$, seen as a $(\hspace{1pt} p+1)$-simplex in $B\bdiffh(M)_\bullet$. More precisely, if we denote this region by $R_\phi$, then
$$
\mathcal R_p(\phi):=\left(R_\phi\cup_{\phi^{-1}}M\times \Delta^p\right)/\sim, \quad (x,0,v)\sim (x,0,w), \quad \forall\ v,w\in \Delta^p, \ x\in M,
$$
where $\phi^{-1}: \phi(M\times\{1\}\times\Delta^p)\overset\cong\longrightarrow M\times \Delta^p$ (see Figure \ref{Rphifigure}). The manifold $\mathcal R_p(\phi)^{d+p+1}$ is stratified over $\Delta^{p+1}$ with $\mathcal R_p(\phi)_0=[M\times\{0\}\times\Delta^p]\cong M$ and $\partial_0\mathcal R(\phi)=[\phi(M\times\{1\}\times\Delta^p)]=M\times\Delta^p$, so it constitutes a $p$-simplex in $\Omega B\bdiffh(M)_\bullet$. Clearly $\mathcal R_\bullet$ is a semi-simplicial map.
\begin{figure}[h]
    \centering
    \includegraphics[scale = 0.12]{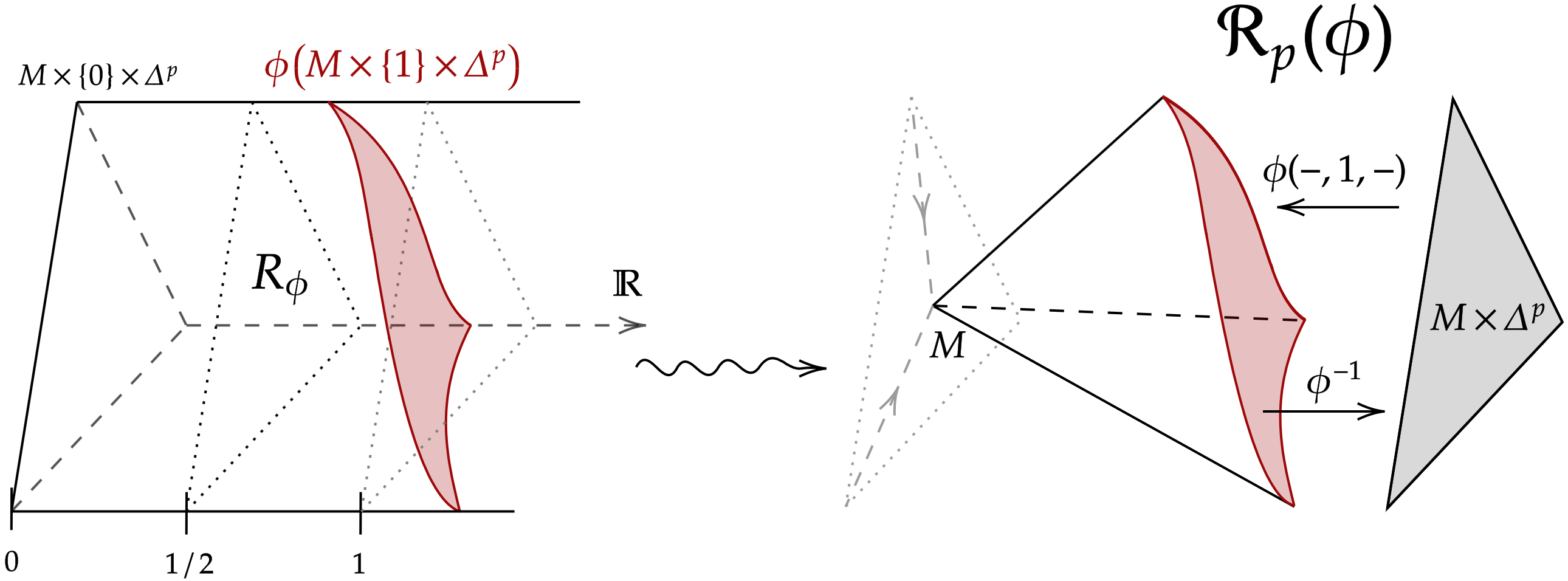}
    \caption{The map $\mathcal R_\bullet$ with $p=2$ and $\dim M=0$.}
    \label{Rphifigure}
\end{figure}

We have to argue that both of the maps in the zig-zag of Proposition \ref{appendixBprop} are equivalences. We begin with the inclusion.

\begin{lem}\label{lemmadiff<1/2}
The inclusion $\bdiffb_{>1/2}(M\times\R)_\bullet\lhook\joinrel\xrightarrow{\ \simeq\ }\bdiffb(M\times \R)_\bullet$ is a weak equivalence.
\end{lem}
\begin{proof}
For a smooth function $\rho: \Delta^p\to \R$, let $T_\rho$ denote the bounded diffeomorphism
$$
T_\rho: M\times \R\times \Delta^p\overset{\cong}\longrightarrow M\times \R\times \Delta^p, \quad (x,t,v)\longmapsto (x,t+\rho(v),v).
$$
We first show that if $\phi\in \bdiffb(M\times \R)_p$ with $\partial_i\phi\in \bdiffb_{>1/2}(M\times \R)_{p-1}$ for all $i=0,\dots, p$, then there exists some $\psi\in \bdiffb_{>1/2}(M\times \R)_p$ with $\partial_i\psi=\partial_i\phi$ for $i=0,\dots,p$ (simplicially) homotopic to $\phi$ in $(\bdiffb(M\times \R)_\bullet, \bdiffb_{>1/2}(M\times \R)_\bullet)$. So let $\phi$ be such a diffeomorphism and set
$$
t_-:=1/2-\min\left\{\mathrm{pr}_\R(\phi(x,1/2,v): x\in M,\ v\in \Delta^p\right\}.
$$
As $\phi$ is continuous, there exists some $\delta>0$ such that for a $\delta$-neighbourhood $B_\delta(\partial\Delta^p)$ of $\partial\Delta^p\subset \Delta^p$,
$$
\phi(M\times (1/2,\infty)\times B_\delta(\partial\Delta^p))\subset M\times (1/2,\infty)\times \Delta^p.
$$
Let $\rho: \Delta^p\to \R_{\geq 0}$ be a smooth cut-off function such that 
$$
\rho\mid_{B_{\delta/2}(\partial \Delta^p)}\equiv 0, \qquad \rho\mid_{\Delta^p\setminus B_\delta(\partial\Delta^p)}\equiv t_-.
$$
Then $\psi:=T_\rho\circ \phi\in \bdiffb_{>1/2}(M\times \R)_p$ is as required. Moreover, the diffeomorphism
$$
T_{(-)\cdot\rho}\circ\phi: (M\times \R\times \Delta^p)\times I\overset{\cong}\longrightarrow (M\times \R\times \Delta^p)\times I, \quad (x,t,v,s)\longmapsto T_{s\cdot\rho}(\phi(x,t,v))
$$
provides the required simplicial homotopy between $\phi$ and $\psi$.

It follows easily from the previous claim that $\pi_p(\bdiffb_{>1/2}(M\times \R)_\bullet)\to \pi_p(\bdiffb(M\times \R)_\bullet)$ is an isomorphism for all $p\geq 0$.
\end{proof}

\begin{lem}
The map $\mathcal R_\bullet$ is a weak equivalence.
\end{lem}
\begin{proof}
There is a map of fibration sequences
$$
\begin{tikzcd}
\Omega B\bdiff(M)_\bullet\dar & &\ar[ll, "M_{(-)}"', "\simeq"] \bdiff(M)_\bullet\dar["{\times\mathrm{Id}_\R}"] &\\
\Omega B\bdiffh(M)_\bullet\dar & &\ar[ll, "\mathcal{R}_\bullet"'] \bdiffb_{>1/2}(M\times \R)_\bullet\dar &\\
\bdiffh/\bdiff(M)_\bullet& &\ar[ll,"{[\mathcal R_\bullet]}"']\bdiffb_{>1/2}(M\times \R)/\bdiff(M)_\bullet\rar[hook, "\simeq"]&\bdiffb(M\times\R)/\bdiff(M)_\bullet.
\end{tikzcd}
$$
The map $M_{(-)}$ is the mapping cylinder construction, so it is an equivalence. In \cite[Cor. 5.5]{WWI} it is shown that the map 
$$
\pi_*([\mathcal R_\bullet]): \pi_*(\bdiffb(M\times \R)/\bdiff(M))\longrightarrow H_*(C_2;\Wh(M))
$$
is injective if $*=0$ and an isomorphism if $*\geq 1$. Clearly the image of $\pi_0([\mathcal R_\bullet])$ lies inside $\frac{I(M)}{\mathcal D(M)}\cong \pi_0(\bdiffh/\bdiff(M))$, as $\mathcal R_0(\phi)$ is an inertial $h$-cobordism for any $\phi\in \diff^b(M\times\R)$. By the five lemma, $\pi_*(\mathcal R_\bullet)$ is an isomorphism for $*\geq 1$, and $\pi_0(\mathcal R_\bullet)$ is injective (note that $\frac{I(M)}{\mathcal D(M)}$ is just a set, but this does not cause any difficulties in the argument).

It remains to show that $\pi_0(\mathcal R_\bullet)$ is surjective. We do this by an \textit{Eilenberg swindle}-like argument as in \cite[Cor. 5.5]{WWI}: namely given an inertial $h$-cobordism $W\in \Omega B\bdiffh(M)_0$, fix two trivialisations (rel the left ends) $W\cup -W\cong M\times [0,1]$ and $-W\cup W\cong M\times [0,1]$. Then there are two different ways of identifying the Eilenberg swindle
$$
S(W):=\dots\cup W\cup -W\cup W\cup\dots=\dots\cup -W\cup W\cup -W\cup\dots
$$
with $M\times \R=\bigcup_{i\in \Z}M\times [i,i+1]$ (in a bounded way). After shifting by an integer, these two identifications give rise to a bounded diffeomorphism $\phi\in \bdiffb_{>1/2}(M\times \R)_0$ such that $\mathcal R_0(\phi)$ is diffeomorphic to $M\times[0,2]\cup_{M\times \{2\}}W$ (see Figure \ref{EilenbergSwindle2}). The homotopy
$$
t\in [0,1]\longmapsto M\times[0,1-2t]\cup_{M\times\{1-2t\}}W
$$
provides a $1$-simplex in $\Omega B\bdiffh(M)_\bullet$ between $\mathcal R_0(\phi)$ and $W$, so $\pi_0(\mathcal R_\bullet)([\phi])=[W]$ as required. We also obtain that $\pi_0(\bdiffb(M\times \R)/\bdiff(M))\cong \frac{I(M)}{\mathcal D(M)}\subset H_0(C_2;\Wh(M))$, which very slightly improves \cite[Cor. 5.5]{WWI}.
\end{proof}

\begin{figure}[h]
    \centering
    \includegraphics[scale = 0.12]{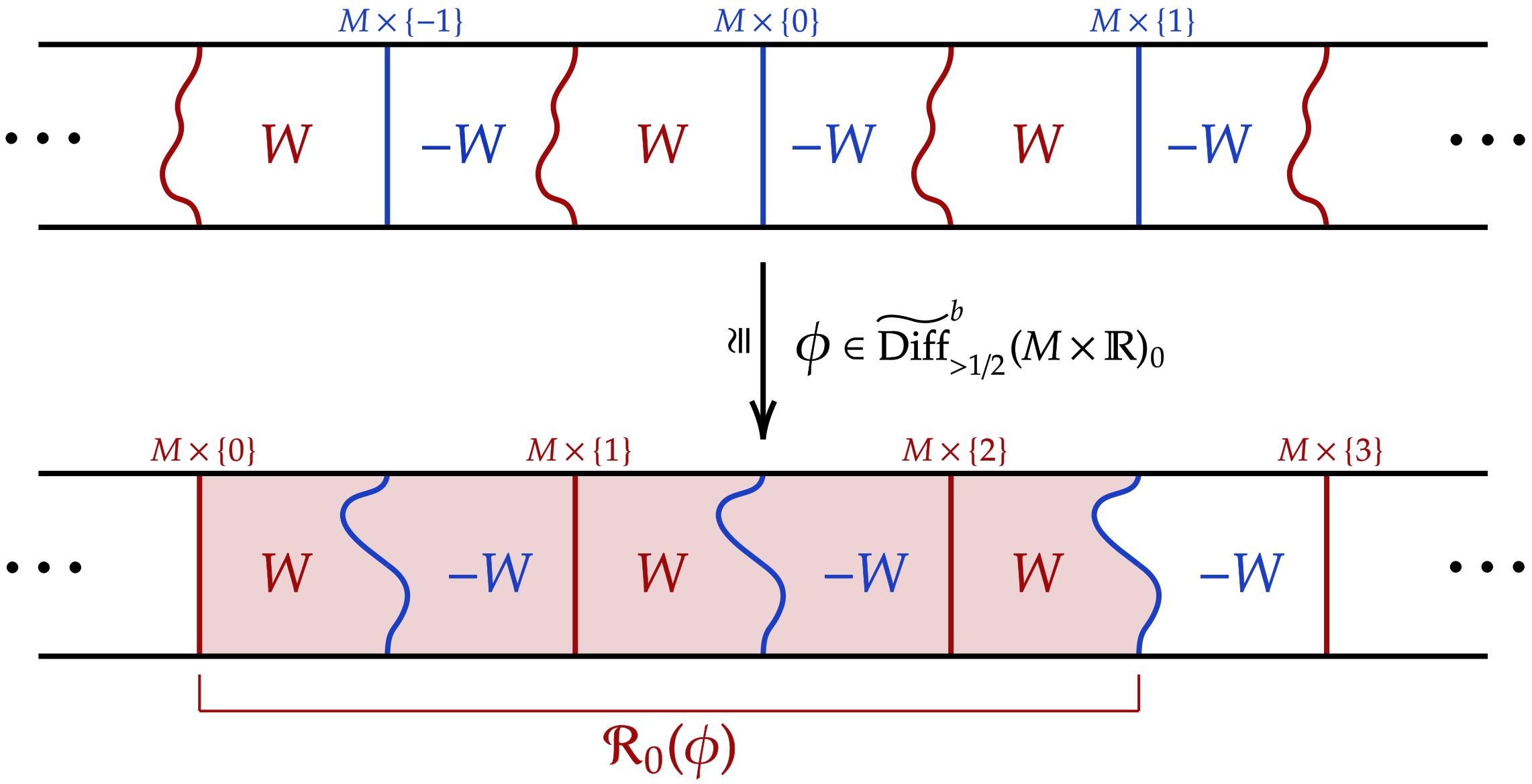}
    \caption{Geometric Eilenberg swindle.}
    \label{EilenbergSwindle2}
\end{figure}

\begin{proof}[Proof of Proposition \ref{appendixBprop}]
Every term in the zig-zag is a Kan complex; the (simplicial) loop space of a Kan complex is Kan, and both $\bdiffb(M\times\R^\infty)_\bullet$ and $\bdiffb_{>1/2}(M\times \R)_\bullet$ are Kan (here we use the collaring condition on face-preserving maps of Definition \ref{stratifieddefn}). Moreover, the simplicial loop space and the Kan loop group are weakly equivalent functors. Since geometric realisations of weak equivalences between Kan complexes are homotopy equivalences, the homotopy equivalence in the second line of the statement follows.
\end{proof}

\subsection{The Whitehead spectrum and Theorem \ref{ThmB} in the context of Weiss--Williams I}\label{whspappendixsection}

For each space $X$, there exists a spectrum $\Whsp(X)$, the \textit{non-connective smooth Whitehead spectrum} of $X$, which recovers the Whitehead group of $X$,
$$
\pi_1^s(\Whsp(X))=\Wh(X).
$$
It is defined to fit in a split\footnote{The splitting is provided by the composition of the Dennis trace map $\mathrm{Tr}: \mathbf{A}(X)\to \Sigma^\infty_+ LX$ postcomposed with the evaluation map $\Sigma^\infty_+ LX\to \Sigma^\infty_+ X$. Note that non-connective $K$-theory is the universal localising invariant in the sense of \cite{BlumbergGepnerTabuada}, so the usual Dennis trace map from \textit{connective} $K$-theory to topological Hochshild homology indeed factors through non-connective $K$-theory.} cofibre sequence of spectra
\begin{equation}\label{WaldhausenWhFibseq}
\begin{tikzcd}
\Sigma^\infty_+X\rar["\iota"] &\mathbf{A}(X)\rar[two heads] &\Whsp(X),
\end{tikzcd}
\end{equation}
where $\mathbf{A}(-)$ denotes Walhausen's non-connective $A$-theory spectrum \cite{WaldhausenAtheory, WaldhausenParametrisedhcob}. The map $\iota$ is the composition of the unit map of $A$-theory $\Sigma^\infty_+X=\mathbb{S}\wedge X_+\to \mathbf{A}(*)\wedge X_+$ and the \textit{assembly map} $\alpha: \mathbf{A}(*)\wedge X_+\to \mathbf{A}(X)$. The topological and piecewise linear versions of the Whitehead spectrum of a space $X$ coincide, and are denoted, slightly abusively, by $\Whsptop(X)$. Explicitly, it fits in a similar cofibre sequence of spectra
\begin{equation}\label{WaldhausenWhFibseqtop}
\begin{tikzcd}
\mathbf{A}(*)\wedge X_+\rar["\alpha"] &\mathbf{A}(X)\rar[two heads] &\Whsptop(X).
\end{tikzcd}
\end{equation}
The Whitehead spectrum is an invariant of the homotopy type of $X$, for both $\iota$ and $\alpha$ are.

With this in mind, let us explain the relation of Theorem \ref{ThmB} to the work of \cite{WWI}. Following the trend of the paper, define the \textit{connective (smooth) $h$-cobordism spectrum} to be
$$
\Hsp^h_\diff(M):=\tau_{\geq 0}\Hsp^h_\diff(M),
$$
the $0$-connective cover of the non-connective version $\Hsp_\diff(M)$. By \cite[Cor. 5.6]{WWI} and \cite[Cor. 5.8]{MElongknots}, it fits in a $C_2$-equivariant fibration sequence of spectra
\begin{equation}\label{Whhfibration}
\begin{tikzcd}
\Hsp^s_\diff(M)\rar& \Hsp^h_\diff(M)\rar & H\Wh(M),
\end{tikzcd}
\end{equation}
where $C_2$ acts on $\Wh(M)$ as in Theorem \ref{ThmB}. In \cite[Thm. B $\&$ C]{WWI} there is established the outer solid square of the homotopy commutative diagram 
$$
\begin{tikzcd}
\bdiff/\diff(M)\ar[rr,"\Phi^s", "\approx"']\dar[hook]& &\Omega^{\infty}\left(\Hsp_\diff^s(M)_{hC_2}\right)\dar\\
\bdiffb(M\times\R)/\diff(M)\overset{\substack{\mathrm{Prop.}\\ \text{\ref{appendixBprop}}}}\simeq\bdiffh/\diff(M)\ar[rr,dashed, "\Phi^h", "\approx_0"']\dar[hook]&&\Omega^{\infty}\left(\Hsp_\diff^h(M)_{hC_2}\right)\dar\\
\bdiffb(M\times \R^\infty)/\diff(M)\simeq\diff^{\hspace{1pt}b}(M\times \R^\infty)/\diff(M)\ar[rr, "\Phi", "\approx_0"']&&\Omega^{\infty}\left(\Hsp_\diff(M)_{hC_2}\right)
\end{tikzcd}
$$
and proved to be homotopy cartesian. The decoration $\approx$ stands for $(\phi_M+1)$-connected, and $\approx_0$ for $(\phi_M+1)$-connected onto the components that are hit, where we recall that $\phi_M$ is the concordance stable range for $M$ (see Theorem \ref{WWbdiffmoddiff}). The existence of the dashed arrow $\Phi^h$ is analogous to that of $\Phi^s$ (in a similar notation as in \cite[$\S$4]{WWI}, replace the filtration of $X:=\diff^{\hspace{1pt}b}(M\times\R^\infty)$ by $\Sigma\mathrm{Filt}_i(X):=\diff^{\hspace{1pt}b}(M\times \R^{i+1})$). The connectivity of $\Phi^h$ can be deduced from that of $\Phi^s$ and $\Phi$. 

Now observe that the composition
\[
\begin{tikzcd}
\Phi^{h/s}:\bdiffh/\bdiff(M)\rar[hook, "\simeq_0"]& {|F_\bullet(M)|} \rar[phantom, "\overset{\text{Thm. \ref{ThmB}}}{\simeq}"]&\Omega^{\infty}\left(H\Wh(M)_{hC_2}\right)
\end{tikzcd}
\]
provides a filler in the diagram of fibre sequences
$$
\begin{tikzcd}
\bdiff/\diff(M)\dar\ar[rrr,"\Phi^s", "\approx"']&&&\Omega^{\infty}\left(\Hsp_\diff^s(M)_{hC_2}\right)\dar\\
\bdiffh/\diff(M)\dar\ar[rrr,"\Phi^h", "\approx_0"']&&&\Omega^{\infty}\left(\Hsp_\diff^h(M)_{hC_2}\right)\dar\\
\bdiffh/\bdiff(M)\ar[dashed, rrr,"\Phi^{h/s}","\simeq_0"']&&&\Omega^{\infty}\left(H\Wh(M)_{hC_2}\right),
\end{tikzcd}
$$
where the right vertical sequence is obtained from (\ref{Whhfibration}) by applying the functor $\Omega^\infty((-)_{hC_2})$. The lower subsquare ought to be homotopy commutative, but we do not have a proof of this claim.

\bibliography{bibliographymcg.bib} 
\bibliographystyle{amsalpha}

\end{document}